\documentclass[12pt,oneside,english]{amsart}

\usepackage{plucker}

\title{Pl\"ucker formulas using equivariant cohomology of coincident root strata}

\author{L\'aszl\'o M. Feh\'er}
\address
{\parbox{\linewidth}{Department of Analysis, Institute of Mathematics, E\"otv\"os Lor\'and University,
\\
P\'azm\'any P\'eter s\'et\'any 1/C, Budapest, Hungary, H-1117
}}
\email{lfeher63@gmail.com}
\author{Andr\'as P. Juh\'asz}
\address
{\parbox{\linewidth}{
    Department of algebraic geometry and differential topology,
  Alfr\'ed R\'enyi Institute of Mathematics,
\\
Re\'altanoda utca 13-15,  Budapest, Hungary, H-1053}}
\email{juhasz.andris@gmail.com}


\begin{document}
\begin{abstract}
We give a new method to calculate the universal cohomology classes of coincident root loci.
 We show a polynomial behavior of them and apply this result to prove that generalized Pl\"ucker formulas are polynomials in the degree, just as the classical Pl\"ucker formulas counting the bitangents and flexes of a degree $d$ generic plane curve.
 We establish an upper bound for the degrees of these polynomials, and we calculate the leading terms
 of those whose degrees reach this upper bound.

We believe that the paper is understandable without detailed knowledge of equivariant cohomology. It may serve as a demonstration of the use of equivariant cohomology in enumerative geometry through the examples of coincident root strata. We also explain how the equivariant method can be \quot{translated} into the traditional non-equivariant method of resolutions.
\end{abstract}
\maketitle
\tableofcontents

\section{Introduction}

This paper grew out from a project of the authors to study motivic invariants like Euler characteristics and $\chi_y$ genus of varieties of tangent lines. At one point we realized that even the \quot{classical} case of enumerating these tangent lines with tangencies of prescribed multiplicities is not fully covered in the literature. We consider the results presented below as a source of conjectures for motivic invariants.

We believe that the paper is understandable without detailed knowledge of equivariant cohomology. It may serve as a demonstration of the use of equivariant cohomology in enumerative geometry through the examples of coincident root strata. In Section \ref{sec:compare} we explain how the equivariant method can be \quot{translated} into the traditional non-equivariant method of resolutions.

We are grateful to Bal\'azs K\H om\H uves for sharing his knowledge and unpublished works with us. Conversations with Rich\'ard Rim\'anyi, Andr\'as N\'emethi, \'Akos Matszangosz and Andrzej Weber also helped and inspired our work. We thank Igor Dolgachev for teaching us about cyclic covers.

The first author was partially supported by NKFI 112703 and 112735 as well as ERC Advanced Grant LTDBud and enjoyed the hospitality of the R\'enyi Institute for a semester.

\subsection{Pl\"ucker numbers}

Let $f\in \Pol^d(\C^n)$ be a nonzero homogeneous polynomial of degree $d$ in $n$ variables. It defines a hypersurface $Z_f=(f=0)$ in $\P(\C^n)$. Let
\[ \lambda=(\lambda_1\geq\lambda_2\geq\cdots\geq\lambda_k)=(2^{e_2},\dots,m^{e_m})\] be a partition without $1$'s and $d\geq|\lambda|=\sum_{i=1}^k \lambda_i$. A line in $\P(\C^n)$ is called a tangent line of type $\lambda$ to $Z_f$ if it has $e_2$ ordinary tangent points, $e_3$ flex points, etc. A formal definition can be given the following way.  Projective lines $[W]$ in $\P(\C^n)$ correspond to affine planes $W^2$ in $\C^n$.
\begin{definition} The projective line $[W]$ is called a \emph{tangent line of type $\lambda$ to $Z_f$} (or
  \emph{$\lambda$-line} for short) if
  \[ f|_W=\prod_{i=1}^{k} \left( f_i^{\lambda_i} \right) \prod_{j=|\lambda|+1}^d \left( f_j\right), \]
where $f_i,f_j:W\to \C$ are linear, and no two of them are scalar multiples of each other.
\end{definition}
For a generic $f \in \Pol^{d}(\mathbb{C}^n)$ and a partition $\lambda$ as above, the set of tangent lines of type $\lambda$ to $Z_f$ is finite when $2(n-2)=\sum_{i=1}^{k}(\lambda_i-1)$.

 Indeed, let $f\in \Pol^d(\C^n)$ and
   \[ \mathcal{T}_\lambda Z_f:=\{ \text{tangent lines of type $\lambda$ to $Z_f$}\} \subset \gr_2(\C^n),\]
the \emph{variety of tangent lines of type $\lambda$ to $Z_f$}. Strictly speaking this is usually not a subvariety but a constructible set, but hopefully it will not cause any confusion that we will use the term subvariety in this broader sense. The dimension of the Grassmannian $\gr_2(\C^n)$---the space of projective lines in $\P(\C^n)$---is $2(n-2)$, and  simple dimension counting gives that the codimension of $\mathcal{T}_\lambda Z_f$ is $\sum_{i=1}^{k}(\lambda_i-1)$. For this reason we introduce the partition
\[\tilde{\lambda}:=(\lambda_1-1,\lambda_2-1,\dots,\lambda_k-1),\]
the \emph{reduction of $\lambda$}. Then
\[\codim\big( \mathcal{T}_\lambda Z_f \subset \gr_2(\C^n)\big)=|\tilde{\lambda}|.\]

This motivates the following.

\begin{definition} Let $\lambda=(2^{e_2},\dots,m^{e_m})$ be a partition without $1$'s such that $2(n_0-2)=|\tilde{\lambda}|$ for some $n_0$. Then the \emph{Pl\"ucker number} $\Pl_\lambda(d)$ for $d\geq|\lambda|$ is defined as the number of type $\lambda$ tangent lines to a generic degree $d$ hypersurface in $\P(\C^{n_0})$.
\end{definition}

\begin{example}   \label{pl22}    The Pl\"ucker numbers
  \[ \Pl_{2,2}(d)=\frac12d(d-2)(d-3)(d+3) \ \ \text{ and } \ \ \Pl_{3}(d)=3d(d-2)  \]
(the number of bitangent lines and flex lines to a generic degre $d$ plane curve)
were calculated by Pl\"ucker in the 1830's. His formulas also include the cases of singular curves, but we only study the generic case.

Note that, for typographical reasons, we omit the brackets from the indices.
\end{example}

 If the dimension of $\mathcal{T}_\lambda Z_f$ is  positive, we can obtain further numbers by adding linear conditions:

 \begin{definition} Let $\lambda=\left( 2^{e_2},\cdots,m^{e_m} \right)$ be a partition without 1's. Choose $n_0$ and $0 \leq i\leq |\tilde\lambda|$ such that $|\tilde\lambda|+i=2(n_0-2)$. We define the \emph{Pl\"ucker number} $\Pl_{\lambda;i}(d)$ for $d\geq|\lambda|$ as the number of $\lambda$-lines of a generic degree $d$ hypersurface in $\P(\C^{n_0})$ intersecting a generic $(i+1)$-codimensional projective subspace.
\end{definition}

For $\Pl_{\lambda;0}(d)$ we recover the previous definition: $\Pl_{\lambda;0}(d)=\Pl_{\lambda}(d)$. We will use both notations.

\begin{example} \label{pl2;1}
For tangent lines we  have
\[ \Pl_{2;1}(d)=d(d-1),\]
the number of lines in $\P(\C^3)$ through a given point and tangent to a generic degree $d$ curve. In other words, the degree of the dual curve is $d(d-1)$.
\end{example}

\begin{example} \label{pl22;2}
For bitangent lines we also have
\[ \Pl_{2,2;2}(d)=\frac{1}{2} d \left( d-1 \right)  \left( d-2 \right)  \left( d-3 \right),\]
 the number of bitangent lines of a generic degree $d$ surface in $\P(\C^4)$ going through a point.
\end{example}

\begin{example}
For $\lambda=(4)$  (the $4$-flexes) we also have two Pl\"ucker numbers:
\[ \Pl_{4;1}(d)=2d(3d-2)(d-3),\ \ \Pl_{4;3}(d)=d(d-1)(d-2)(d-3), \]
where $\Pl_{4;1}(d)$ is the number of 4-flex lines to a generic degree $d$ surface in $\P(\C^4)$ intersecting a line, and $\Pl_{4;3}(d)$ is the number of 4-flex lines to a generic degree $d$ hypersurface in $\P(\C^5)$ going through a point.
\end{example}

\begin{remark}\label{rmrk_ndi}
  Notice that $n_0$ doesn't appear in our notation:
  $\Pl_{\lambda;i}(d)$ is defined as a number of certain $\lambda$-lines in $\mathbb{P}(\mathbb{C}^{n_0})$
  for a specific $n_0$ that is determined by $\lambda$ and $i$ via $|\tilde\lambda|+i=2(n_0-2)$.
  (This also shows that the parity of admissible $i$'s is fixed: $i=|\tilde{\lambda}|,
  |\tilde{\lambda}|-2,\dots$, in particular, $\Pl_{\lambda;0}(d)=\Pl_{\lambda}(d)$ is defined
for a partition $\lambda$ only if $|\tilde \lambda|$ is even.)

  Moreover, $\Pl_{\lambda;i}(d)$ solves a family of enumerative problems: Choose an
  $n \geq n_0$. Elementary geometric considerations imply that
  $\Pl_{\lambda;i}(d)$ is the number of $\lambda$-lines of a generic degree $d$ hypersurface
  in $\mathbb{P}(\mathbb{C}^n)$ intersecting a generic $(n-n_0+i+1)$-codimensional projective subspace $A$
  and contained in a generic $(n_0-1)$-dimensional projective subspace $B$ such that $A \subset B$.
\end{remark}

\subsection{Calculating Pl\"ucker numbers} \label{sec:calc-of-pl}
The key observation is that for a given $\lambda$, all the Pl\"ucker numbers $\Pl_{\lambda;i}(d)$ are encoded in the cohomology class
\[ \left[\,\overline{\mathcal{T}_\lambda Z_f} \subset \gr_2(\C^n) \right]   \in H^*\big(\gr_2(\C^n)\big) \]
for any $n\geq |\tilde\lambda|+2$ and $f\in \Pol^d(\C^n)$ generic.

Let  $s_{k,l}$ for $l\leq k\leq n-2$ denote the Schur polynomials, $s_1=c_1,\ s_2=c_1^2-c_2,\ s_{1,1}=c_2,\ s_{2,1}=c_1c_2$, etc., where $c_1,\ c_2$ are the Chern classes of $S^\vee\to \gr_2(\C^n)$, the dual of the tautological rank two bundle over the Grassmannian $\gr_2(\C^n)$. Then $\{s_{k,l}:l\leq k\leq n-2\}$ is a basis of $H^*(\gr_2(\C^n))$ with dual basis $\{s_{n-2-l,n-2-k}:l\leq k\leq n-2\}$. The Schur polynomial $s_{k,l}$ is the cohomology class of the Schubert variety
\begin{equation*}\label{Schubert-variety}
  \sigma_{k,l}=\big\{W\in \gr_2(\C^n): \dim (W\cap F_{n-k-1})\geq 1 \text{ and } W\subset F_{n-l}\big\},
\end{equation*}
\noindent where $F_i\subset \C^n$ is the subspace spanned by the first $i$ coordinate vectors. This implies that the Schur coefficients of $\left[\,\overline{\mathcal{T}_\lambda Z_f} \subset \gr_2(\C^n) \right]$ are solutions of enumerative problems: Standard transversality argument implies that if
\[ \left[\,\overline{\mathcal{T}_\lambda Z_f} \subset \gr_2(\C^n) \right]=\sum a_{k,l}s_{k,l}, \]
then $a_{k,l}$ is the number of $\lambda$-lines in $\sigma_{n-2-l,n-2-k}$.
Setting $(k,l)=(|\tilde{\lambda}|-j,j)$, we see that being in $\sigma_{n-2-l,n-2-k}$ is equivalent to
the linear conditions of Remark \ref{rmrk_ndi} for $i=|\tilde{\lambda}|-2j$. This implies

\begin{proposition}  \label{T-and-pluecker} Let $\lambda=(2^{e_2},\dots,m^{e_m})$ be a partition without $1$'s, $n\geq |\tilde\lambda|+2$ and $f\in \Pol^d(\C^n)$ generic. Then
\begin{equation*}
  \left[\,\overline{\mathcal{T}_\lambda Z_f}\subset \gr_2(\C^n) \right] =\sum_{j=0}^{t}
                 \Pl_{\lambda;|\tilde\lambda|-2j}(d)    s_{|\tilde\lambda|-j,j},
\end{equation*}
where $t=\lfloor|\tilde\lambda|/2\rfloor$.
\end{proposition}

Consequently, for a
given $\lambda=(2^{e_2},\dots,m^{e_m})$ calculating all the
Pl\"ucker numbers for $\lambda$ is equivalent to calculating the cohomology class
$\left[\,\overline{\mathcal{T}_\lambda Z_f}\subset \gr_2(\C^n) \right]$ for any $n \geq
|\tilde{\lambda}|+2$ and $f \in \Pol^{d}(\mathbb{C}^n)$ generic.

Note that---corresponding to Remark \ref{rmrk_ndi}---this cohomology class is stable in the sense that in the Schur basis it is independent of $n\geq |\tilde\lambda|+2$. This justifies omitting $n$ from the Schur coefficients of $\left[\,\overline{\mathcal{T}_\lambda Z_f}\subset \gr_2(\C^n) \right]$ and writing $a_{k,l}$.
Taking the formal limit $n\to\infty$, we obtain a unique polynomial in $\Z[c_1,c_2]$. The ring $\Z[c_1,c_2]$ can be identified with the $\GL(2)$-equivariant cohomology ring of the point $H^*_{\GL(2)}:=H^*(\bor\GL(2))$, and this polynomial is the equivariant cohomology class $[\, \overline{Y}_\lambda(d)]_{\GL(2)}$ of the coincidence root stratum $\overline{Y}_\lambda(d)$, what we will define in Section \ref{sec:coh}.

This is the reason why in this paper we first give an algorithm to calculate these equivariant cohomology classes, then we study the implications on the behaviour of the Pl\"ucker numbers.

\begin{remark} One can work with the Chow rings instead of cohomology. For the spaces appearing in the paper they are isomorphic. \end{remark}
\subsection{Results and structure of the paper}

The key result is that the Pl\"ucker numbers $\Pl_{\lambda;i}(d)$ are polynomials in $d$. This was our motivation for having $d$ as a variable in our notation. We also give several structural results on these polynomials.

In Section \ref{sec:coh} we explain how to calculate the Pl\"ucker numbers using the equivariant cohomology classes $\left[\,\overline Y_\lambda(d)\right]$ of the coincident root strata $\,\overline Y_\lambda(d)\subset\Pol^d(\C^2)$. This connection makes Theorem \ref{recursion4Y} the key technical result of the paper: we give an inductive formula for $\left[\,\overline Y_\lambda(d)\right]$, where the induction is on the length of the partition $\lambda$. Several formulas were already known for the  equivariant classes $\left[\,\overline Y_\lambda(d)\right]$, however those are less suited for our purposes. A detailed account of those formulas is given in Section \ref{sec:early}.

Section \ref{sec:proof} is devoted to the proof of Theorem \ref{recursion4Y}. Most of the paper is independent of this section, except parts of Sections \ref{sec:compare} and \ref{sec:further}.

In Section \ref{sec:poly} we show---using Theorem \ref{recursion4Y}---that the equivariant classes of $\,\overline Y_\lambda(d)$ are polynomials of degree $|\lambda|$ in $d$ (Theorem \ref{Y-poly}) and, consequently, the Pl\"ucker numbers $\Pl_{\lambda;i}(d)$ are polynomials of degree at most $|\lambda|$ in $d$ (Theorem \ref{pl-poly}).
Furthermore, in Theorem \ref{fundclass-leading} we calculate the leading term
(the $d$-degree $|\lambda|$ part) of $\left[\,\overline Y_\lambda(d)\right]$.
We also deduce a simple closed formula for a large class of Pl\"ucker numbers:
\medskip

\noindent\textbf{Theorem \ref{plpoint}.}
\emph{
Let $\lambda=\left( 2^{e_2},\cdots,m^{e_m} \right)$ be a partition without 1's. Then
\[\Pl_{\lambda;|\tilde{\lambda}|}=
\coef\!\big(s_{|\tilde\lambda|},\left[\,\overline{Y_\lambda}(d)\right]\big)=
\frac1{\prod_{i=2}^m \left( e_i! \right)}d(d-1)\cdots(d-|\lambda|+1).\]
In other words, for $n = |\tilde{\lambda}|+2$ we calculated the number of
$\lambda$-lines for a generic degree $d$
hypersurface in $\mathbb{P}(\mathbb{C}^{n})$ through a generic point of
$\mathbb{P}(\mathbb{C}^n)$.
}
\medskip

\noindent
Finally, we state Theorem \ref{pl-la-drop2} that tells us the $d$-degrees of all
the Pl\"ucker numbers.

In Section \ref{sec:m-flexes} we restrict our attention to $\lambda=(m)$ and
give closed formulas for the Pl\"ucker numbers $\Pl_{m;i}(d)$.
As a consequence we get a closed formula for a classical problem:
\medskip

\noindent\textbf{Theorem \ref{hyperflex-number}.}
\emph{A generic degree $d=2n-3$ hypersurface in $\P(\C^n)$ possesses
  \[ \sum_{u=1}^{n-1}(-1)^{u+n+1}\stir du\binom{d-u+1}{n-1}d^u\]
  lines which intersect the hypersurface in a single point. Here $\stir du$ denotes the  Stirling number of the first kind.}
\medskip

\noindent
For $n=3$ it says that a generic cubic plane curve has 9 flexes. For $n=4$ we obtain the classical result that a generic quintic has 575 lines which intersect the hypersurface in a single point.\\
In Section \ref{subsec_linesonhypersurface} we connect Pl\"ucker numbers $\Pl_{d;i}(d)$ to enumerative problems
regarding the number of lines on degree $d$ hypersurfaces. In particular, we give a new proof of
Don Zagier's formula (\cite{grunberg-moree-zagier}) on the number of lines on a degree $d=2n-3$ hypersurface
in $\P(\C^{n+1})$. This connection also implies that
\medskip

\noindent\textbf{Theorem \ref{lines-on-vs-hyperflexes}.}
\emph{The number of lines on a generic degree $d=2n-3$ hypersurface in $\P(\C^{n+1})$ is $d$ times the number of hyperflexes of a generic degree $d$ hypersurface in $\P(\C^{n})$.}
\medskip

\noindent
We expected this to be a classical result, but found no mention of it in the literature.

In Section \ref{sec:poly-of-Pl} we calculate the coefficient of $d^{|\lambda|}$ in $\Pl_{\lambda;i}(d)$ by relating it to certain Kostka---and for special $\lambda$'s
to Catalan and Riordan---numbers.
This coefficient informs us about the asymptotic behaviour of the  Pl\"ucker number $\Pl_{\lambda;i}(d)$ as $d$ tends to infinity, so we will call it the \emph{asymptotic  Pl\"ucker number} $\apl_{\lambda;i}$.
The main theorem of the section is
\medskip

\noindent\textbf{Theorem \ref{thm:apl}.}
\emph{Let $\lambda=(2^{e_2},\dots,m^{e_m})$ be a partition without $1$'s and $j\leq \lfloor |\tilde{\lambda}|/2 \rfloor$ a nonnegative
 integer. Let $n=|\tilde{\lambda}|-j+2$.Then}
 \[ \apl_{\lambda;|\tilde{\lambda}|-2j}= \frac{K_{(n-2,j),\tilde{\lambda}}}{\prod_{i=2}^m \left( e_i! \right)}, \]
where $K_{\mu,\nu}$ denotes the Kostka numbers.
\medskip

In Section \ref{sec:compare}, using the example of $m$-flex lines, we compare our method with the classical non-equivariant approach.
We introduce the notion of incidence varieties and we use them to formulate a
non-equivariant version of Theorem \ref{recursion4Y}.
We try to convince the readers who are not familiar with the equivariant method, that it is a useful language which can be translated to classical terms.

In Section \ref{sec:further} we study variants of the Pl\"ucker numbers. In Section \ref{sec:correspondence} we show how a substitution into the cohomology class $\left[\,\overline Y_\lambda(d)\right]$ calculates Pl\"ucker numbers for linear sytems of hypersurfaces. A small example is the number of flex lines through a point to a pencil of degree $d$ curves. In Section \ref{sec:mflexes-of-lambdalines} we study the variety of $m$-flex points of $\lambda$-lines. To demonstrate the versatility of the method we give the details for
computing the degree of the curve of flex points of the $(3,2)$-lines to a surface. In Section \ref{sec:combined} we show that the previous two constructions can be combined without difficulty.
As an example we calculate the degree of the curve of tangent points of bitangent lines to a pencil of degree $d$ plane curves.

In Appendix \ref{sec:trans} we collect the transversality results needed.

\section{Cohomological calculations} \label{sec:coh}
\subsection{The variety of tangent lines as a coincident root locus} \label{sec:variety-of-tangent}

In this section we identify $\mathcal{T}_\lambda Z_f$ as a coincident root locus of a certain section of a vector bundle.
\bigskip

The vector space
\[\Pol^d(\C^2):=\{\mbox{homogeneous polynomials of degree $d$ in two variables}\}\]
admits a stratification into the so-called coincident root strata (CRS).

\begin{definition} Let $\lambda=(2^{e_2},\dots,m^{e_m})$ be a partition without $1$'s and $d\geq |\lambda|$. Then the \emph{coincident root stratum} of $\lambda$ is
  \[   Y_\lambda(d)  :=  \left\{ f\in \Pol^d(\C^2) :  f=\prod_{i=1}^{k} \left( f_i^{\lambda_i} \right)
  \prod_{j=|\lambda|+1}^d \left(  f_j \right) \right\},    \]
where $f_i,f_j:\C^2 \to \C$ are nonzero, linear and no two of them are scalar multiples of each other.
\end{definition}

Here we slightly changed the usual notation, since we are interested in the $d$-dependence of these strata. If $d$ is clear from the context, we will also use the shorthand notation $Y_\lambda$.
Note that the above definition includes for all $d$'s the open stratum $Y_\emptyset(d)$ corresponding to the
empty partition $\emptyset$.
The strata $Y_\lambda(d)$ together with $\{0\}$ gives a stratification of $\Pol^d(\C^2)$.

Their key property is that they are invariant for the $\GL(2)$-action on
$\Pol^d(\C^2)\cong \Sym^d\left({\mathbb{C}^2}^\vee \right)$ coming from the standard
representation of $\GL(2)$ on $\mathbb{C}^2$.
It is intuitively clear that the codimension of $Y_\lambda(d)$ in $\Pol^d(\C^2)$ is $|\tilde{\lambda}|=\sum_{i=1}^{k}(\lambda_i-1)=\sum_{j=2}^{m}(j-1)e_j$ since every increase of the multiplicity of a root by one increases the codimension by one. For details see e.g.  \cite{fnr-root}.

\bigskip

We will study the $Y_\lambda(d)$-points of  vector bundles $\pol^d(D)$ for various rank two vector bundles $D$. For this reason, we introduce the notation:
\begin{definition} Let the Lie group $G$ act on the vector space $V$, and let $E=P\times_GV$ be a vector bundle associated to the principal $G$-bundle $P$. Then for any $G$-invariant subvariety
  $Y \subset V$ let
\[ Y(E):=P\times_GY\subset E.\]
\end{definition}

We apply this construction to $S\to \gr_2(\C^n)$, the tautological rank  two bundle over the Grassmannian. Given a nonzero homogeneous polynomial $f\in \Pol^d(\C^n)$ we can define a section $\sigma_f(W):=f|_W$ of the vector bundle $\Pol^d(S) \to \gr_2(\C^n)$. Then, by definition,
\begin{equation*}\label{var-of-tang-as-locus}
  \mathcal{T}_\lambda Z_f=\sigma_f^{-1}\big(Y_\lambda(d)(\Pol^d(S))\big).
\end{equation*}

We will also use the shorthand notation $Y_\lambda(\Pol^d(S))$ for $Y_\lambda(d)(\Pol^d(S))$.

\subsection{Calculating the cohomology class}\label{subsec:calccohomclass}
\subsubsection{Equivariant cohomology classes of invariant subvarieties} \label{sec:equi-class-def}
Suppose that the algebraic Lie group $G$ acts on an algebraic manifold $M$ and $Y\subset M$ is a $k$-codimensional $G$-invariant subvariety. Then we can define the \emph{$G$-equivariant cohomology class of $Y$}:
\[  [Y\subset M]_G\in H^{2k}_G(M).\]
This class was defined by several people independently and by quite different methods. Our approach is the closest to \cite{totaro}. In this paper $G$ is always the product of general linear groups, in which case
the construction of the  $G$-equivariant cohomology class is simpler:

Suppose that $G=\GL(r)$. We define an \emph{approximation of the universal bundle} $E\GL(r)\to \B\GL(r)$ as $P\to \gr_r(\C^N)$, where $P$ is the frame bundle of the tautological bundle of the Grassmannian $\gr_r(\C^N)$. Then $B:=P\times_{\GL(r)}M$ approximates the Borel construction $\B_{\GL(r)}M$ in the sense that the map $\beta: B \to \B_{\GL(r)}M$---induced by the classifying map of $P$---induces an isomorphism $\beta^*:H^{2k}_{\GL(r)}(M)=H^{2k}(\B_{\GL(r)}M) \to H^{2k}(B)$
for every sufficiently large $N$. Therefore, we can define
  \begin{equation}\label{equi-class-def}
  [Y\subset M]_{\GL(r)}:=(\beta^*)^{-1}[P\times_{\GL(r)}Y\subset P\times_{\GL(r)}M],
  \end{equation}
and it is not difficult to see that this definition is independent of the choice of $N$ big enough.
For products of $\GL(r_i)$'s we can use the products of the approximations.

If it is clear from the context, we drop the group $G$ or the ambient space from our notation and write
$\left[ Y \right]$ or $\left[ Y \right]_G$ for $\left[ Y \subset M \right]_G$. Similarly, we sometimes drop
the group from equivariant characteristic classes of $G$-bundles $E \to B$ and write $c_i(E)$ for
$c_i^G(E)$ and $e(E)$ for $e_G(E)$.

\bigskip

A fundamental observation in equivariant cohomology is that equivariant cohomology classes of invariant subvarieties are universal polynomials: for example for the $\GL(2)$-invariant subvariety $\overline Y_\lambda(d)$
and a rank two vector bundle $D$
 we have
  \begin{equation*}\label{equi}
    \left[\,\overline Y_\lambda(\Pol^d(D)) \subset \Pol^d(D)\right]=\left[\,\overline Y_\lambda(d)
    \subset \Pol^d(\mathbb{C}^2) \right]_{\GL(2)}|_{c_i\mapsto c_i(D^\vee)}.
  \end{equation*}
For more details on the choice and interpretation of the generators $c_1,c_2$ see Section
\ref{sec:conventions}.

\medskip

For a generic homogeneous polynomial $f\in \Pol^d(\C^n)$ the section $\sigma_f$ is transversal to $\,\overline Y_\lambda(\Pol^d(S))$, see Section \ref{alg-bertini}, implying that
  \[  [\sigma_f^{-1}\big(\,\overline Y_\lambda(\Pol^d(S))\big) \subset \gr_2(\C^n)]=\sigma_f^*\left[\,\overline Y_\lambda(\Pol^d(S)) \subset \Pol^d(S)\right], \]
  where the pullback $\sigma_f^*:H^*(\Pol^d(S))\to H^*(\gr_2(\C^n))$ is an isomorphism. This isomorphism  is independent of $f$, so we will not denote it in our formulas.
\begin{corollary} \label{cor:equi2nonequi}
The cohomology class $\left[\,\overline{\mathcal{T}_\lambda Z_f}\subset \Gr_2(\mathbb{C}^n)\,\right]$ is obtained from the equivariant class $\left[\,\overline Y_\lambda(d)\right]_{\GL(2)}\in \Z[c_1,c_2]$ by substituting $c_i(S^\vee)$ into $c_i$ for $i=1,2$.
\end{corollary}

\begin{remark} There is a subtle detail about the preimage of the closure. Our definition for transversality is that $\sigma_f$ has to be transversal to all strata $Y_\mu(\Pol^d(S))$. Since the coincident root stratification satisfies the Whitney regularity condition, transversality implies that
\[ \sigma_f^{-1}\big(\,\overline Y_\lambda(\Pol^d(S))\big)=\,\overline{\sigma_f^{-1}\big( Y_\lambda(\Pol^d(S))\big)}=\,\overline{\mathcal{T}_\lambda Z_f}.\]
However, usage of the Whitney property is not needed. Let $f:M\to N$ be an algebraic map of smooth varieties and assume that $f$ is transversal to the subvariety $X\subset N$, in the sense that it is transversal to some stratification $X=\coprod X_i$ with $X_0$ being the open stratum. Then it is possible that $\,\overline{f^{-1}(X_0)}$ is strictly smaller than ${f^{-1}(X)}$, but the difference is a union of components of smaller dimension, so the cohomology classes $\left[\,\overline{f^{-1}(X_0)}\right]$ and $[{f^{-1}(X)}]=f^*[X]$ agree.
\end{remark}

Now, we can rephrase Proposition \ref{T-and-pluecker}:
\begin{proposition}  \label{Y-and-pluecker} Let $\lambda=(2^{e_2},\dots,m^{e_m})$ be a partition without $1$'s. Then
\begin{equation*}
  \left[\,\overline Y_\lambda(d)\right] =\sum_{j=0}^{t}
                 \Pl_{\lambda;|\tilde\lambda|-2j}(d)    s_{|\tilde\lambda|-j,j},
\end{equation*}
where $t=\lfloor|\tilde\lambda|/2\rfloor$.
\end{proposition}

This connection motivates our calculations of the equivariant classes $\left[\,\overline Y_\lambda(d)\right]$.

\begin{remark} As we mentioned in the introduction, we can avoid referring to equivariant cohomology here. First, observe that for the embedding $i:\gr_2(\C^n)\to \gr_2(\C^{n+1})$ the equality
\[
i^*\left[ \left( \,\overline{Y}_\lambda\left( \Pol^d(S) \right) \subset \Pol^{d}(S) \right)
\to \Gr_2\left( \mathbb{C}^{n+1} \right) \right]=
\left[ \left( \,\overline{Y}_\lambda\left( \Pol^d(S)\right) \subset \Pol^{d}(S) \right)
\to \Gr_2\left( \mathbb{C}^{n} \right) \right]
\]
holds. Also, notice that $H^*(\gr_2(\C^n))=\Z[c_1,c_2]/I_n$, where the degree of generators of $I_n$ tends to infinity with $n$. This implies the existence and uniqueness of a polynomial $\left[\,\overline Y_\lambda(d)\right]_{\GL(2)}\in \Z[c_1,c_2]$ with the property above.

For a general rank two vector bundle $D\to M$ over a projective algebraic manifold we can use the fact that any such bundle can be pulled back from $S\to \gr_2(\C^{n})$ for $n>>0$. This argument can be generalized to obtain a general definition of the $G$-equivariant cohomology class of a $G$-invariant subvariety of a vector space $V$, where $G$ is an algebraic group acting on $V$ (see e. g. \cite{totaro}).

\end{remark}

\subsection{A recursive formula for \texorpdfstring{$\left[\,\overline Y_\lambda(d)\right]$}{[Ylambda(d)]}}
The main result of this section is Theorem \ref{recursion4Y}, which gives an algorithm to calculate the universal cohomology classes $\left[\,\overline Y_\lambda(d)\right]$.

The class $\left[\,\overline Y_\lambda(d)\right]\in\Z[c_1,c_2]$ can be expressed in the \emph{Chern roots} $a$ and $b$: substituting $c_1\mapsto a+b$ and $c_2\mapsto ab$, we obtain a polynomial symmetric in the variables $a$ and $b$.

\begin{theorem}\label{recursion4Y} Let $\lambda=(2^{e_2},\dots,m^{e_m})$ be a partition without $1$'s and $d\geq |\lambda|$.
Let $\lambda'$ denote the partition $(2^{e_2},\dots,m^{e_m-1})$, where $e_m=1$ is allowed.  We also use the notation  $d'=d-m$. Then
\[ \left[\,\overline Y_\lambda(d)\right]=\frac{1}{e_m}\partial\Big(\left[\,\overline Y_{\lambda'}(d')\right]_{m/d'}\prod_{i=0}^{m-1}\big( ia+(d-i)b \big) \Big),\]
where for a polynomial $\alpha\in\Z[a,b]$ and $q\in \Q$ we use the notation
\[ \alpha_{q}(a,b)=\alpha(a+qa,b+qa)\]
and
  \[  \partial(\alpha)(a,b)=\frac{\alpha(a,b)-\alpha(b,a)}{b-a}\]
  denotes the divided difference operation.
\end{theorem}

The notation $d'=d-m$ will be used throughout this paper.
\begin{remark}
For any given $d$ the class $\left[\,\overline Y_\lambda(d)\right]$ is in $\Z[c_1,c_2]$, which is not obvious from the recursion formula because of the divisions.
\end{remark}
\begin{example} For $\lambda=(m)$ we recover the formula of \cite[Ex. 3.7 (4)]{fnr-root}:
 \begin{equation}\label{Y_m}
   \left[\,\overline Y_m(d)\right]=\partial \left( \prod_{i=0}^{m-1}\big( ia+(d-i)b \big)  \right).
\end{equation}
For example,
 \begin{equation*}\label{Y2}
\begin{split}
 \left[\,\overline Y_2(d)\right]=&\partial \Big( \big(db\big)\big(a+(d-1)b\big)\Big)=d\frac{(ab+(d-1)b^2)-(ba+(d-1)a^2)}{b-a}\\
=& d(d-1)(a+b)=d(d-1)c_1=d(d-1)s_1
\end{split}
 \end{equation*}
 and
 \begin{equation*}\label{Y3}
   \begin{split}
  \left[\,\overline{Y}_3(d))\right] =&  d(d-1)(d-2)c_1^2-d(d-2)(d-4)c_2 \\
                           =&  d(d-2)(d-1)s_2+3d(d-2)s_{1,1}.
\end{split}
 \end{equation*}
\end{example}

\begin{example}
For $\lambda=\left( 2,2 \right)$, $m=2$, $\lambda'=\left( 2 \right)$ and $d'=d-2$.
Hence we have
\[ \left[\,\overline Y_2(d-2)\right]_{2/d-2}=(d-2)(d-3)\left(a+\frac{2}{d-2}a+b+\frac{2}{d-2}a\right)=(d-3)\big((d+2)a+(d-2)b\big), \]
implying that

\begin{equation}\label{Y22}
\begin{split}
\left[\,\overline Y_{2,2}(d)\right]&=\frac12 \partial\Big( (d-3)\big((d+2)a+(d-2)b\big)db\big(a+(d-1)b\big) \Big) \\
&=\frac12d(d-3)\partial\Big(b(a+(d-1)b)\big((d+2)a+(d-2)b\big)\Big) \\
&=\frac12d(d-3)(d-2)\Big((d-1)c_1^2+4c_2\Big) \\
&=\frac12d \left( d-1 \right)  \left( d-2 \right)  \left( d-3 \right) s_{{
2}}+\frac12d \left( d-2 \right)  \left( d-3 \right)  \left( d+3
 \right) s_{{1,1}},
\end{split}
\end{equation}
which is a calculation still manageable by hand. Notice that the result is in agreement with Examples \ref{pl22} and \ref{pl22;2}.

Also, notice that we obtained these results for all $d$'s at the same time, and the polynomial dependence is obvious. This is true for any partition $\lambda$, which will be proved in Section \ref{sec:poly}.
\end{example}

The recursion formula to calculate these polynomials is easy to implement for example in Maple, and it is fast: for $|\lambda|<40$ the results are immediate on a PC.
\subsection{Earlier formulas} \label{sec:early}
 Using Kleiman's theory of multiple point formulas (\cite{kleiman1977enumtheorysingularities,
kleiman1981multiplepoint_iteration,Kleiman1982multiplepoint_formaps}) Le Barz in \cite{lebarz-formules} and Colley  in \cite{colley1986contact} calculated examples of Pl\"ucker numbers.

Kirwan gave formulas for the $\SL(2)$-equivariant cohomology classes of coincident root strata in \cite{kirwan}. Notice that the $\SL(2)$-equivariant cohomology classes are obtained from the $\GL(2)$-equivariant ones by substituting zero into $c_1$, therefore they do not determine the corresponding Pl\"ucker numbers.
The first formula for the $\GL(2)$-equivariant cohomology classes $\left[\,\overline Y_\lambda(d)\right]$ was given in \cite{fnr-root}. 
 Soon after a different formula was calculated with different methods in \cite{balazs-tezis}. These formulas don't seem to be useful for proving polynomiality in $d$. In 2006 in his unpublished paper \cite{kazarian}  Kazarian deduced a formula in a form of a generating function from his theory of multisingularities of Morin maps based on Kleiman's theory of multiple point formulas. This formula shows the polynomial dependence but further properties doesn't seem to follow easily. He also calculates several Pl\"ucker numbers $\Pl_\lambda(d)$. The paper \cite{spink-tseng} of Spink and Tseng also develops a method to calculate the $\GL(2)$-equivariant cohomology classes $\left[\,\overline Y_\lambda(d)\right]$. One of their main goals is to establish relations between these classes.

\section{Proof of the recursion formula} \label{sec:proof}

The proof of Theorem \ref{recursion4Y} is based on the following fundamental property of the equivariant cohomology class:
\begin{lemma} \label{cover}
  Let $f:M\to N$ be a proper $G$-equivariant map of smooth varieties with $\tilde Y\subset M$. Suppose that $f|_{\tilde Y}$ is generically $k$-to-1 to its image $Y\subset N$. Then
  \[ [Y\subset N]=\frac{1}{k}f_![\tilde  Y\subset M].  \]
\end{lemma}
We will apply Lemma \ref{cover} to the projection
\[ \pi: \P^1\times\Pol^{d}(\C^2)\to \Pol^{d}(\C^2) \]
and $Y=\overline{Y}_\lambda(d)$.

\begin{remark} \label{projective-cover}
To motivate the following construction of $\tilde Y$, let us look at a projective version: we construct an $e_m$-fold branched covering of $\P\overline{Y}_\lambda\subset \P\Pol^{d}(\C^2)\iso \P^d$. Consider the map
  \[ f:\P^1\times\P^{d'}\to \P^d,\]
where $f=\mu\circ\big(v\times\Id_{\P^{d'}}\big)$, $v:\P^1\to \P^m$ is the Veronese map and $\mu:\P^m\times\P^{d'}\to \P^d$ is the projectivization of the multiplication map $\pol^m(\C^2)\times \pol^{d'}(\C^2)\to \pol^d(\C^2).$ Then it is not difficult to see that $f|_{\P^1\times \P\overline{Y}_{\lambda'}}$ is generically $e_m$-to-1 to its image $\P\overline{Y}_\lambda(d)$.

Indeed, $\P^d$ can be identified with the space of unordered $d$-tuples of points of $\P^1$ with multiplicities. The map $f$ corresponds to adding an extra point with multiplicity $m$ to a $d'$-tuple of points, and a $d$-tuple of multiplicity $\lambda$ has $e_m$ preimages, depending on which point of multiplicity $m$ comes from the $\P^1$ factor.

To obtain our $\tilde Y$ we need to \quot{deprojectivize} this construction. It is possible to use this projective construction to prove the recursion formula, but the expressions for the equivariant cohomology rings and the pushforward maps are more complicated.
\end{remark}

\subsection{The construction of the covering space: twisting with a line bundle}\label{subsec:twistingwithlinebundle}
\begin{definition} \label{scalar-extension} For any representation $\rho:G \to \GL(V)$ of a Lie group $G$ on a vector space $V$, we define its \emph{scalar extension} $\tilde{\rho}: G \times \GL(1) \to \GL(V)$ as the tensor product of $\rho$ and $\id_{\GL(1)}$ with $V \otimes \mathbb{C} \cong V$ identified canonically.

Let $Y \subset V$ be a $\rho$-invariant subvariety, not necessarily closed. If $Y$ is a \emph{cone}, i.e. invariant for the scalar $\GL(1)$-action on $V$, then it is also $\tilde{\rho}$-invariant.

Now, if $A= P \times_\rho V \to M$ is a vector bundle associated to $P$ and $L \to M$ is any line bundle, then using its frame bundle $L^{\times}=\Inj(\mathbb{C},L)$, we can obtain $A\otimes L$ as a bundle associated to the principal $G \times \GL(1)$ bundle $P \times_M L^{\times}\to M$:
\[ A \otimes L = ( P \times_M L^{\times})\times_{\rho \otimes\id_{\GL(1)}} ( V \otimes \mathbb{C}) \cong ( P \times_M L^{\times})\times_{\tilde{\rho}} V. \]
This description allows us to define a subvariety of $A \otimes L$
\begin{equation*}
  \begin{split}
     Y(A \otimes L) := & \big(P \times_M L^{\times}\big) \times_{\tilde{\rho}} Y          \\
                     = & \left\{ e_m \otimes l_m: e_m \in (P \times_\rho Y)_m, l_m \in L_m \setminus \left\{ 0 \right\}, m \in M \right\}.
  \end{split}
\end{equation*}
\end{definition}

Our primary example is the tensor product
\[ E:=\Pol^{d'}(\mathbb{C}^2) \otimes \Pol^{m}(\mathbb{C}^2/\gamma) =
  \left( (\P^1 \times \GL(2))\times_{\P^1} \Pol^{m}(\mathbb{C}^2/\gamma)^{\times}\right)
\times_{\tilde{\rho}} \Pol^{d'}(\mathbb{C}^2) \]
---where $\gamma$ is the tautological line bundle over $\P(\mathbb{C}^2)$ and $\rho: \GL(2) \to
\Pol^{d'}(\mathbb{C}^2)$ is our usual representation---
and its subvariety
\[ \tilde Y:=  \overline{Y}_{\lambda'}(d')\left( \Pol^{d'}(\mathbb{C}^2) \otimes \Pol^{m}(\mathbb{C}^2/\gamma)\right) \subset E \]
given by the cone $\overline{Y}_{\lambda'}\left( d' \right) \subset \Pol^{d'}(\mathbb{C}^2)$.

We have an injective map
\[ j: E=\Pol^{d'}(\mathbb{C}^2) \otimes \Pol^{m}(\mathbb{C}^2/\gamma ) \hookrightarrow \P^1\times \Pol^{d}(\mathbb{C}^2)\]
induced by the multiplication of polynomials:
\[ j(f\otimes g)(v):=\big(V, f(v)\cdot g(v+V)\big), \]
 where $f\in \Pol^{d'}(\mathbb{C}^2)$ and $V<\mathbb{C}^2$ is the one-dimensional subspace such that $g\in \Pol^{m}(\mathbb{C}^2/V)$. Therefore, we consider $E$ and $\tilde Y$ to be subspaces of $ \P^1\times \Pol^{d}(\mathbb{C}^2)$.

 The projection $\pi:\P^1 \times\Pol^{d}(\C^2)\to \Pol^{d}(\C^2)$ restricted to $\tilde Y$ is generically $e_m$-to-1 to its image $\overline{Y}_\lambda(d)$, implying that
 \begin{equation}\label{eq:push1}
   \left[\, \overline{Y}_\lambda(d)\subset \Pol^{d}(\mathbb{C}^2) \right]=\frac{1}{e_m}\pi_!\Big[\tilde Y\subset \P^1\times\Pol^{d}(\mathbb{C}^2)\Big].
 \end{equation}

 Notice that all the maps above are $\GL(2)$-equivariant, so we consider all these cohomology classes and the pushforward equivariantly.

 \begin{remark}
   Notice that all these varieties admit compatible $\GL(1)$-actions induced by the scalar multiplication.
   Omitting the zero sections and factoring out by this $\GL(1)$-action, we recover the construction of Remark \ref{projective-cover}.
 \end{remark}

 An easy argument gives that
 \begin{lemma} \label{push-from-subbundle}
  Let $E\to M$ be a subbundle of the vector bundle $\hat{E}\to M$. Then the pushforward map induced by  the inclusion $j:E\to \hat{E}$ is given by
  \[ j_!z= z \cdot e(\hat{E}/E), \]
where we did not denote the  isomorphisms $j^*:H^*(\hat{E})\iso H^*(E)$ and $H^*(\hat{E})\iso H^*(M)$. Equivariant versions of the statement also hold.
\end{lemma}
The lemma implies that
 \begin{equation}\label{eq:product}
   \Big[\tilde Y\subset \P^1\times\Pol^{d}(\mathbb{C}^2)\Big]=e\left(\big(\P^1\times\Pol^{d}(\mathbb{C}^2)\big)/E\right)\cdot [\tilde Y\subset E].
 \end{equation}
 The key step in the proof of Theorem \ref{recursion4Y} is the calculation of $[\tilde Y\subset E]$, which we will do in the next sections.

\subsection{Conventions} \label{sec:conventions} To be able to make these calculations we need to fix generators of the cohomology rings involved.
Most calculations of the paper happen in $H^*(\B\GL(2))$, the $\GL(2)$-equivariant cohomology ring of the point.

Our goal is to obtain \quot{positive} expressions, so we choose $c_i=c_i(S^{\vee})$ as generators of $H^*(\B\GL(2))=\Z[c_1, c_2]$,
where $S=E\GL(2)\times_{\GL(2)}\C^2$ is the tautological rank two bundle over the infinite Grassmannian $\gr_2(\C^\infty) \simeq \B\GL(2)$. We will also use the \quot{Chern roots}: Let $\T$ denote the subgroup of diagonal matrices in $\GL(2)$. The complex torus $\T$ is isomorphic to $\GL(1)^2$. The inclusion $i:\T\to\GL(2)$ induces an injective homomorphism $\B i^*: H^*(\B\GL(2))\to H^*(\B\T)\iso\Z[a,b]$ with image the symmetric polynomials in the variables $a$ and $b$. Let $\pi_i:\T\to \GL(1)$ denote the projection to the $i$-th factor and $L_i:=E\T\times_{\pi_i}\C$ denote the tautological line bundles over the factors of $\B\T \simeq \P(\C^\infty)\times\P(\C^\infty)$. To be consistent with our first choice we use the notation $a:=c_1(L_1^\vee)$ and $b:=c_1(L_2^\vee)$, so $\B i^*(c_1)=a+b$ and $\B i^*(c_2)=ab$.

For equivariant cohomology we need to specify (left) group actions on the spaces we are interested in. Our convention is that the $\GL(2)$-action on $\C^2$ is the standard action. This induces a $\GL(2)$-action on $\Pol^d(\C^2)$ via $(gp)(v):=p(g^{-1}v)$. We obtain $\T$-actions by restriction. These choices imply that the $\T$-equivariant Chern class of $\C^2$ is $c^{\T}(\C^2)=(1-a)(1-b)$, i.e. the \emph{weights} of $\C^2$ are $-a$ and $-b$. Also, $c^{\T}\big(\Pol^d(\C^2)\big)=\prod_{i=0}^{d}\big(1+ia+(d-i)b\big)$, i.e. the weights of $\Pol^d(\C^2)$ are $db,a+(d-1)b,\dots,da$.

The standard action of $\GL(2)$ on $\C^2$ induces an action on $\P(\C^2)$. Its restriction to $\T$ has fixed points $\langle e_1\rangle$ and $\langle e_2\rangle$, where $\C^2 = \langle e_1,e_2 \rangle$.
We will need the equivariant Euler classes of the tangent spaces of these fixed points:
\begin{equation*}
  e^{\T}\left(T_{\langle e_1\rangle} \P(\C^2)\right)=e^{\T}\big(\Hom(\langle e_1\rangle,\langle e_2\rangle)\big)=(-b)-(-a)=a-b, \ \ \ \ \
e^{\T}\left(T_{\langle e_2\rangle} \P(\C^2)\right)=b-a.
\end{equation*}

With these choices the formulas are nicer. We pay the price in the proof of Theorem \ref{recursion4Y}, where the signs will change several times.

\subsection{The twisted class}
The results of this section are based on \cite[\S 6.]{fnr-forms}. A special case of the twisted class appeared earlier in \cite{harris-tu}  under the name of squaring principle.

As $\tilde{Y}=\overline{Y}_{\lambda'}(d') (\Pol^{d'}(\mathbb{C}^2) \otimes \Pol^{m}(\mathbb{C}^2/\gamma))$
can be defined as a bundle associated to a principal $\GL(2) \times
\GL(1)$-bundle using the $\tilde \rho$-action on $\overline{Y}_{\lambda'}(d')$,
we can---using the universal property of the equivariant class---compute its cohomology class
$[ \tilde Y \subset E ]$ from $\left[ \, \overline Y_{\lambda'}(d')\right]_{\GL(2) \times \GL(1)}$.

For our representation $\rho:\GL(2)\to \GL\big(\Pol^{d'}(\mathbb{C}^2)\big)$ and invariant subvariety $Y_{\lambda'}(d') \subset \Pol^{d'}(\mathbb{C}^2)$ it is
possible to calculate $[\,\overline Y_{\lambda'}(d')]_{\GL(2) \times \GL(1)}$ from $[ \,\overline Y_{\lambda'}(d')]_{\GL(2)}$.

\bigskip
More generally, let us say that a representation $\rho: G \to \GL(V)$ of a complex reductive group $G$ \emph{ contains the scalars} if there is a homomorphism $\varphi: \GL(1) \to G$ and a positive integer $d$ such that
\begin{equation}\label{eq_containsthescalars}
  \rho\left( \varphi(s) \right)=s^d \id_V .
 \end{equation}
 For such $\varphi$, $G$ has a maximal complex torus $\mathbb{T}^r \subset G$ with $\im(\varphi) \subset \mathbb{T}^r$ and
 $\varphi: \GL(1) \to \mathbb{T}^r\cong \GL(1)^r$ can be written as $\varphi(s)=\left( s^{w_1},\dots,s^{w_r} \right)$ for some integers $w_i$.
 For our representation $\Pol^{d'}(\mathbb{C}^2)$, we can choose $w_1=w_2=-1$ and $d=d'$.

In this paper we are only concerned with group actions of the general linear group $G=\GL(r)$,
in which case we can---without limiting generality---restrict the actions to a maximal complex torus
$i: \mathbb{T}^r\iso\GL(1)^r \hookrightarrow G$.  We choose $a_i:=c_1\big(E\mathbb{T}^r\times_{\pi_i}\C\big)$ as generators of $H^*(\B \mathbb{T}^r)$, where the homomorphism $\pi_i:\mathbb{T}^r\to \GL(1)$ is the projection to the $i$-th factor. This is the most common choice for generators. Compared with the conventions of Section \ref{sec:conventions}, we have $a=-a_1$ and $b=-a_2$.

By the splitting principle, the induced map
$i^*: H^*(\B\GL(r); \mathbb{Z}) \to H^*(\B \mathbb{T}^r; \mathbb{Z})$ is an isomorphism onto its image
$H^*(\B \mathbb{T}^r; \mathbb{Z})^{S_r} \cong \mathbb{Z}[a_1,\dots,a_r]^{S_r}$
such that $ i^* \left[ Y \right]_G=\left[ Y \right]_{\mathbb{T}}$ for any $G$-invariant subset $Y \subset V$.

For a general connected Lie group $G$---by the Borel theorem---the analogous isomorphism holds with rational coefficients onto $H^*(BT, \mathbb{Q})^{\mathcal{W}}$, where $T$ is a real maximal torus of $G$ and  $\mathcal{W}$ is the Weyl group of $G$. This implies that the results of this section can be easily generalized to connected Lie groups.
\bigskip

For the following discussion it will be convenient to keep track of not only the groups acting but the actions themselves. For this reason, the
$G$-equivariant class of a subvariety $Y \subset V$ invariant under the $G$-action $\rho:G\to \GL(V)$ will be denoted by $\left[ Y \right]_\rho$. 
\begin{proposition} \label{twisted} Suppose that the representation $\rho: \GL(r) \to \GL(V)$
  contains the scalars as above. If $Y\subset V$ is a $\rho$-invariant closed subvariety,
  then it is also invariant for the scalar extension $\tilde{\rho}$ (see Definition \ref{scalar-extension}), and
   \[   [Y]_{\tilde{\rho}}(a_1,\dots,a_r,x)=[Y]_\rho(a_1+\frac{w_1}{d}x,\dots,a_r+\frac{w_r}{d}x),\]
   where $\left[ Y \right]_\rho \in H^*_{\GL(r)} \cong\Z[a_1,\dots,a_r]^{S_r}$
   and $H^*_{\GL(1)}\iso\Z[x]$ such that $x=c_1\big(E\GL(1)\times_{1_{\GL(1)}}\C\big)$.
\end{proposition}
\begin{proof} We can restrict the $\GL(r)$-action to the maximal torus $\T^r$ without losing information. We use the same notation $\rho$ for the restriction to $\T^r$.
\smallskip

Let $\sigma:\T^r\times\GL(1)\to \T^r$ and $\psi:\T^r\times\GL(1)\to \T^r\times\GL(1)$ denote the homomorphisms
  \[ \sigma(t_1,\dots,t_r,s)  =  \varphi(s)\cdot(t_1^d,\dots,t_r^d)  =  (s^{w_1}t_1^d,\dots,s^{w_r}t_r^d)  \]
  and
  \[ \psi(t_1,\dots,t_r,s)  =  (t_1^d,\dots,t_r^d,s^d).  \]
Then \eqref{eq_containsthescalars} and the definition of $\tilde{\rho}$ imply that
$\rho\circ \sigma=\tilde{\rho}\circ \psi$. Equivariant cohomology is functorial in the $G$ variable.
This means that for any $\rho$-invariant subvariety $Y\subset V$ we have
 $\psi^*[Y]_{\tilde{\rho}}=\sigma^*[Y]_{\rho}$.\\
Since $\sigma^*(a_i)=d a_i + w_i x$, $\psi^*(a_i)=d a_i$ and $\psi^*(x)=d x$
with $a_i, x$ chosen as above, for $\left[ Y \right]_{\tilde{\rho}} \in
H^*_{ \T^r\times \GL(1)} \cong \mathbb{Z}\left[ a_1,\dots,a_r,x \right]$ we have
  \[ [Y]_{\tilde{\rho}}(da_1,\dots,da_r,dx)  =  [Y]_{\rho}(da_1+w_1x,\dots,da_r+w_rx).\]
  Since $[Y]_{\tilde{\rho}}$ and $[Y]_{\rho}$ are homogeneous polynomials of the same degree---the codimension $c$ of $Y \subset V$---, we can divide by $d^c$,
which implies the proposition.
\end{proof}

The universal property of the equivariant class $\left[ Y \right]_{\tilde{\rho}}$ immediately implies
\begin{corollary}\label{twist4bundle}
  Let $A=P \times_\rho V \to M$ for some principal $\GL(r)$-bundle $P \to M$. Suppose that the representation $\rho: \GL(r) \to \GL(V)$ contains the scalars as above and that
$Y \subset V$ is a $\rho$-invariant closed subvariety. Let $L \to M$ be a line bundle. Then for
the subvariety $Y(A \otimes L)$ defined in Definition \ref{scalar-extension} we have
\[ [ Y(A \otimes L) \subset A \otimes L ]=[ Y ]_\rho( \alpha_1+\frac{w_1}{d}\xi,\dots,\alpha_r+\frac{w_r}{d}\xi ),\]
where $[ Y ]_\rho \in H^*_{\GL(r)} \cong  \mathbb{Z}[ a_1,\dots,a_r]^{S_r}$, $\alpha_1,\dots,\alpha_r$ and
$\xi$ are
the Chern roots of $P$ and $L$.
\end{corollary}
\begin{proof} Using the splitting principle, we can replace $P$ with a principal $\T^r$-bundle. Indeed, we have the splitting manifold $\hat{M}:=\fl(P\times_{\GL(r)}\C^r)$ with the property that the projection $p:\hat{M}\to M$ induces an injective homomorphism $p^*:H^*(M)\to H^*(\hat{M})$ and $p^*P$ can be reduced to a principal $\T^r$-bundle. We can also pull back the bundles $A$ and $L$ to $\hat{M}$. Therefore, in the rest of the proof let $P \to M$ denote a principal $\T^r$-bundle.

As $Y( A \otimes L ) = (P \times_M L^{\times}) \times_{\tilde{\rho}} Y$, ---by the universal property of
$\left[ Y \right]_{\tilde{\rho}}$ and Proposition \ref{twisted}---we have
  \[ \left[ Y(A \otimes L) \subset  A \otimes L \right]=
  \tilde{\kappa}^* \left[ Y \right]_{\tilde{\rho}}=\tilde{\kappa}^* \left(
[Y]_\rho\big(a_1+\frac{w_1}{d}x,\dots,a_r+\frac{w_r}{d}x\big) \right)
, \]
    where  $\tilde{\kappa}: M \to \B\left(\T^r\times \GL(1) \right)=
    \B \T^r \times \B \GL(1) $ is the classifying map of $P \times_M L^{\times}$. We complete the proof by noticing that
    $\tilde{\kappa}^{*} a_i=\alpha_i$ and $\tilde{\kappa}^* x=\xi$.
\end{proof}

\begin{corollary}\label{cor:equivclass_of_twisted}
Let $\rho:\T \to \GL(V)$ be a representation of the torus $\T=\T^r$ that contains the scalars, and let $Y \subset V$ be a $\rho$-invariant closed subvariety. Let $L'\to M'$ be a $\T$-line bundle over the $\T$-space $M'$. Then $V \otimes L' \to M'$ is also a $\T$-vector bundle via the diagonal action. Then for the $\T$-invariant subvariety $Y(V\otimes L')$ we have
  \[ \left[ Y(V \otimes L') \subset V \otimes L' \right]_{\T}=
           [Y]_{\rho}\left( a_1+\frac{w_i}{d}c_1^{\T}( L'),\dots,a_r+\frac{w_r}{d}c_1^{\T}( L') \right),\]
where $[Y]_\rho \in  H^*_{\T}\cong\Z[a_1,\dots,a_r]$ and $c_1^{\T}(L')\in H^*_{\T}(M')$ is the $\T$-equivariant first Chern class of $L'$.

\end{corollary}
\begin{proof} Recall from Section \ref{sec:equi-class-def} that we can approximate $E\T\to \B\T$ with
   \[  P:=(\C^N\setminus 0)^r\to \P(\C^N)^r,   \]
 and for $\beta:P\times_{\T}M'\to \B_{\T}M'$ we have
   \[  \beta^*[Y(V\otimes L')\subset V\otimes L']_{\T}=   [P\times_{\T}\big(Y(V\otimes L')\big)\subset P\times_{\T}(V\otimes L')].  \]
 For $M:=P\times_{\T}M'$, $A:=p^*(P\times_{\T}V)$ for the fibration $p:M\to \P(\C^N)^r$, and $L:=P\times_{\T}L'$ Corollary \ref{twist4bundle} implies that
 \[  [P\times_{\T}\big(Y(V\otimes L')\big)\subset P\times_{\T}(V\otimes L')]=\left[Y(A\otimes L)\subset A\otimes L\right]=
 [Y]_\rho(\alpha_1+\frac{w_1}d\xi,\dots,\alpha_r+\frac{w_r}d\xi),\]
 where the $\alpha_i$'s are Chern roots of $p^*P$ and $\xi=c_1(L)$.
 Using that $\beta^*a_i=\alpha_i$, $\beta^*c_1^{\T}(L')=\xi$ and that $\beta^*$ is injective,
 we obtain the result.
\end{proof}

\medskip

Now, we are able to compute $[\,\tilde Y\subset E]$: Set $w_1=w_2=-1$ and $d=d'$ corresponding
to our $\GL(2)$-representation $\Pol^{d'}(\mathbb{C}^2)$.
In accordance with our conventions $a_1=-a$ and $a_2=-b$,
\begin{equation}\label{eq_c1TPolmQ}
c_1(\Pol^m(\mathbb{C}^2/\gamma))=-m\left( -(a+b)-c_1(\gamma) \right)=
m\left( a+b+c_1(\gamma) \right).
\end{equation}
Then, by Corollary \ref{cor:equivclass_of_twisted},
$[\tilde Y\subset E]=
[\, \overline{Y}_{\lambda'}(d')( \Pol^{d'}(\mathbb{C}^2) \otimes \Pol^{m}(\mathbb{C}^2/\gamma))
\subset \Pol^{d'}(\mathbb{C}^2) \otimes \Pol^{m}(\mathbb{C}^2/\gamma)]$
can be obtained from $[\, \overline{Y}_{\lambda'}(d') ]$ by substituting
\[ -a\mapsto -a+\frac{-1}{d'}\big(m(a+b+c_1(\gamma))\big)\ \ \text{and} \ \ -b\mapsto -b+\frac{-1}{d'}\big(m(a+b+c_1(\gamma))\big),\]
or, multiplying by $-1$, by substituting
\begin{equation}\label{eq_tildeYclasssubs}
a\mapsto a+\frac{m}{d'}\big(a+b+c_1(\gamma)\big)\ \ \text{and} \ \ b\mapsto b+\frac{m}{d'}\big(a+b+c_1(\gamma)\big).
\end{equation}

\subsection{The pushforward map \texorpdfstring{$\pi_!$}{pi!}}
According to the Atiyah-Bott-Berline-Vergne (ABBV) integral formula, for a cohomology class $\alpha\in H^*_{\T^2}\big(\P^1\times\Pol^d(\C^2)\big)$ its pushforward along
$\pi: \mathbb{P}^1 \times \Pol^d(\mathbb{C}^2) \to \Pol^d(\mathbb{C}^2)$ is
\[\pi_!\alpha=\int\limits_{\P^1} \alpha=
\frac{\alpha|_{\langle e_1\rangle}}{e(T_{\langle e_1\rangle} \P^1)}+\frac{\alpha|_{\langle e_2\rangle}}{e(T_{\langle e_2 \rangle} \P^1)}=
\frac{\alpha|_{\langle e_1\rangle}}{a-b}+\frac{\alpha|_{\langle e_2\rangle}}{b-a}, \]
where $\langle e_1\rangle$ and $\langle e_2\rangle$ are the fixed points of the torus $\T^2$ of diagonal matrices acting on $\P^1=\P(\C^2)$, and $-a,\ -b$ are the weights of $\langle e_1\rangle$ and $\langle e_2\rangle$, according to the conventions of Section \ref{sec:conventions}.

If $\alpha \in H^*_{\GL(2)}(\mathbb{P}^1)$ and $q(a,b):=\alpha|_{\langle e_2\rangle}$, then
$\alpha|_{\langle e_1\rangle}=q(b,a)$, therefore
\begin{equation}\label{equiv-integral}
\int\limits_{\P^1} \alpha=\partial (q).
\end{equation}
\bigskip

\begin{proof}[Proof of Theorem \ref{recursion4Y}]
  By \eqref{eq:push1} and \eqref{eq:product},
  \[
    [\, \overline{Y}_\lambda(d)\subset \Pol^{d}(\mathbb{C}^2)]=
    \frac{1}{e_m}\pi_!\Big[\tilde Y\subset \P^1\times\Pol^{d}(\mathbb{C}^2)\Big]=
    \frac{1}{e_m} \pi_! \left(  e\left(\big(\P^1\times\Pol^{d}(\mathbb{C}^2)\big)/E\right)\cdot [\tilde Y\subset E] \right),
  \]
where $E=\Pol^{d'}(\mathbb{C}^2) \otimes \Pol^{m}(\mathbb{C}^2/\gamma) \to \P^1$.
  The pushforward can be computed as above, therefore, to complete the proof we have to determine the
  restriction of
  \[ \alpha=e\left(\big(\P^1\times\Pol^{d}(\mathbb{C}^2)\big)/E\right)\cdot [\tilde Y\subset E] \]
  to $\langle e_2\rangle$. Restricting to $\langle e_2 \rangle$ amounts to substituting $c_1(\gamma) \mapsto
  -b$, hence using \eqref{eq_c1TPolmQ}, we have
  \begin{equation*}
    e(\Pol^d(\C^2)/E)|_{\langle e_2\rangle}=
    \left.
      \frac{ \prod_{i=0}^{d} \big( ia+(d-i)b \big) }
    {\prod_{i=0}^{d'} \big( m (c_1(\gamma)+a+b)+ia+(d'-i)b \big)}
    \right|_{c_1(\gamma)\mapsto-b}=
    \prod_{i=0}^{m-1}\big( ia+(d-i)b \big) ,
    \end{equation*}
    and by \eqref{eq_tildeYclasssubs},
    \begin{equation} \label{eq:shift}
    [\tilde Y\subset E]|_{\langle e_2\rangle}=\left[\,\overline Y_{\lambda'}(d')\right]_{m/d'},
  \end{equation}
  where we used the notation of Theorem \ref{recursion4Y}.
\end{proof}

\bigskip

\begin{remark}
Note that the left-hand side of (\ref{eq:shift}) can also be interpreted as  the equivariant cohomology class of the $\T$-invariant subvariety $x^m \overline Y_{\lambda'}(d') $ of $x^m \Pol^{d'}(\mathbb{C}^2)$:
\begin{equation*}
\left.\left[\tilde Y\subset E\right]\right|_{\langle e_2 \rangle}=
\left[ x^m \,\overline Y_{\lambda'}(d') \subset x^m \Pol^{d'}(\mathbb{C}^2) \right]_{\T},
\end{equation*}
where $y,x \in \Pol^{1}(\mathbb{C}^2)$ denotes the dual basis of $e_1,e_2$.
\end{remark}

\section{Polynomiality of \texorpdfstring{$\left[\,\overline Y_\lambda(d)\right]$}{[Ylambda(d)]}} \label{sec:poly}

\begin{theorem} \label{Y-poly} The classes $\left[\,\overline Y_\lambda(d)\right]$ are polynomials in $d$: $\left[\,\overline Y_\lambda(d)\right]\in \Q[c_1,c_2,d]=
  \Q[a,b]^{S_2}[d]$.
\end{theorem}
\begin{proof}
We use the following well-known statement:
\begin{lemma}\label{rat-poly} Suppose that $q(x)$ is a rational function, such that $q(d)$ is an integer for all $d>>0$ integers. Then $q(x)$ is a polynomial. \end{lemma}
The lemma can be proved by induction on the degree of $q(x)$: Notice that $q(x+1)-q(x)$ has the same property but smaller degree than $q(x)$.
\bigskip

We prove the theorem by induction on the length of the partition $\lambda$ using Theorem \ref{recursion4Y}. Suppose that we already know that $\left[\,\overline Y_{\lambda'}(d')\right]$ is a polynomial in $d'=d-m$. Then all the coefficients (of the monomials $a^ib^j$) of $\left[\,\overline Y_{\lambda'}(d-m)\right]_{m/(d-m)}$ are rational functions. The values of these coefficients are integers for $d>>0$
since, by (\ref{eq:shift}), $\left[\,\overline Y_{\lambda'}(d-m)\right]_{m/(d-m)}$ is an equivariant cohomology class of an invariant subvariety.
Then Lemma \ref{rat-poly} implies that the class $\left[\,\overline Y_{\lambda'}(d-m)\right]_{m/(d-m)}$ is a polynomial in $d$. The class $e(\Pol^d(\C^2)/E)|_{\langle e_2\rangle}=\prod_{i=0}^{m-1}\big( ia+(d-i)b \big) $ is clearly a polynomial in $d$, and the divided difference operator preserves polynomiality in $d$.
\end{proof}
Using Proposition \ref{Y-and-pluecker}, we see that Theorem \ref{Y-poly} is equivalent to
\begin{theorem}\label{pl-poly}
  The Pl\"ucker numbers $\Pl_{\lambda;i}(d)$ for $0\leq i\leq |\tilde\lambda|$ and $i\equiv |\tilde\lambda|$
  $(\!\!\!\! \mod 2)$ are polynomials in $d$: there is a unique polynomial $p(d)$ such that $\Pl_{\lambda;i}(d)=p(d)$ for $d\geq|\lambda|$.
\end{theorem}
\begin{remark}
By definition, these polynomials have integer values for $d>>0$, therefore for every integer. \end{remark}
\bigskip

\subsection{The leading term of \texorpdfstring{$\left[\,\overline Y_\lambda(d)\right]$}{{[Ylambda(d)]}}}
The asymptotic behaviour of the classes $\left[\,\overline Y_\lambda(d)\right]$ is determined
by the largest $d$-degree parts. Using the same inductive argument, we can find out what these $d$-leading
terms are.
\begin{theorem}\label{fundclass-leading}
  For any $\lambda=\left( 2 ^{e_2},\cdots,m^{e_m} \right)$ the top $d$-degree part of $\left[\,\overline Y_\lambda(d)\right] \in \Q[a,b]^{S_2}[d]$ is
  \[ \frac{1}{\prod_{i=2}^m \left( e_i!\right)} h_{\tilde{\lambda}} d^{|\lambda|}, \]
where $h_\nu$ is the \emph{complete symmetric polynomial} corresponding to the partition $\nu=(\nu_1,\dots,\nu_k)$: $h_\nu=\prod h_{\nu_i}$
 with $h_i$  the $i$-th complete symmetric polynomial in $\{a,b\}$.
\end{theorem}

\begin{proof} First, notice that $\left[\,\overline Y_\emptyset(d)\right]=1$, proving the codimension $0$ case.

 Now, let $\lambda=\left( 2^{e_2},\cdots,m^{e_m} \right)$ and consider the divided difference formula of Theorem \ref{recursion4Y}.
  The induction hypothesis says that the $d$-leading term of $\left[\,\overline Y_{\lambda'}(d)\right]$ is
  \[ \frac{1}{\prod_{i=2}^{m-1} \left(
   e_i!  \right) \left( e_m -1\right)!} h_{\tilde{\lambda'}}
  d^{|\lambda|-m}.\]
 The class $\left[\,\overline Y_{\lambda'}(d-m)\right]_{m/(d-m)}$ has the same
 $d$-leading term, in particular, it is symmetric.
 The largest $d$-degree part of $\prod_{i=0}^{m-1}\big(ia+(d-i)b\big)$ is $d^m b^m$, hence the $d$-leading term of
 $\left[\,\overline Y_\lambda(d)\right]$ is
 \[ \frac{1}{e_m} \frac{1}{\prod_{i=2}^{m-1}\left( e_i!  \right)
   \left( e_m -1\right)!}  h_{\tilde{\lambda'}} d^{|\lambda|-m} \frac{d^m(b^m-a^m)}{b-a}
   =\frac{1}{\prod_{i=2}^{m}\left( e_i!\right)} h_{\tilde{\lambda'}} h_{m-1} d^{|\lambda|}=
 \frac{1}{\prod_{i=2}^{m}\left( e_i!\right)} h_{\tilde{\lambda}}  d^{|\lambda|}.\]
\end{proof}
In particular, we have that
\begin{theorem}  \label{deg=la} The polynomial $\left[\,\overline Y_{\lambda}(d)\right]\in \Q[c_1,c_2][d]$ has $d$-degree $|\lambda|$.
\end{theorem}

\subsection{\texorpdfstring{$\Pl_{\lambda;|\tilde{\lambda}|}$}{Pllambda;|lambdatilde|} and \texorpdfstring{$d$}{d}-degrees of Pl\"ucker numbers}
The recursive argument in Theorem \ref{recursion4Y} can also be used to calculate $\Pl_{\lambda;|\tilde{\lambda}|}(d)$ for an arbitrary partition $\lambda$:

\begin{theorem}\label{plpoint}
Let $\lambda=\left( 2^{e_2},\cdots,m^{e_m} \right)$ be a partition without 1's. Then
\[\Pl_{\lambda;|\tilde{\lambda}|}=
  \coef\!\big(s_{|\tilde\lambda|},\left[\,\overline{Y}_\lambda(d)\right]\big)=
\frac1{\prod_{i=2}^m (e_i!)}d(d-1)\cdots(d-|\lambda|+1).\]
In other words, for $n = |\tilde{\lambda}|+2$ we calculated the number of
$\lambda$-lines for a generic degree $d$
hypersurface in $\mathbb{P}(\mathbb{C}^{n})$ through a generic point of
$\mathbb{P}(\mathbb{C}^n)$.
\end{theorem}
\begin{proof} We use induction on $|\lambda|$. Write the restriction of the twisted class, \eqref{eq:shift},
  as
\begin{equation}\label{eq_Ylambdasubstituted_as_sum}
\left[ \,\overline{Y}_{\lambda'}(d')\right]_{m/d'}=
\sum_{t=0}^{|\widetilde{\lambda'}|} \left( \frac{m}{d'}a \right)^t q_t(a,b,d') ,
\end{equation}
where $q_t \in \mathbb{Q}\left[ a,b,d' \right]$ is symmetric in $a,b$. Note that
$q_0=\left[ Y_{\lambda'}(d') \right]$.

As $m \ge 2$
\[ \prod_{i=0}^{m-1} \big( ia+(d-i)b \big) =d(d-1)\dots(d-m+1)b^m+ab\cdot p(a,b,d)\]
for some $p \in \mathbb{Z}[a,b,d]$. Then
\begin{equation*}
  \begin{split}
    \left[\,\overline{Y}_\lambda(d)\right]
    &=\frac{1}{e_m} \partial\left( \left[ \,\overline{Y}_{\lambda'}(d') \right]_{m/d'}
    \prod_{i=0}^{m-1}  \big( ia+ (d-i)b \big)  \right) 
    \\
    &=\frac{1}{e_m} \partial\left(
      \sum_{t=0}^{|\widetilde{\lambda'}|} \left( \left( \frac{m}{d'}a \right)^t q_t(a,b,d')\right)
    \big( d(d-1)\dots(d-m+1)b^m+ab\cdot p(a,b,d) \big) \right)
    \\
    &=\frac{1}{e_m} \partial\big( q_0(a,b,d') \cdot d (d-1)\dots (d-m+1) b^m  + ab \cdot r(a,b,d)
    \big)\\
    &= \frac{1}{e_m} \left[ \,\overline{Y}_{\lambda'}(d') \right] d (d-1) \dots(d-m+1) \partial(b^m) +
    \frac{1}{e_m} ab\cdot \partial\big( r(a,b,d) \big)
  \end{split}
\end{equation*}
for some $r \in \mathbb{Q}\left[ a,b,d \right]$. Using the induction hypothesis,
\[ \left[ \,\overline{Y}_{\lambda'}\left( d' \right) \right] =\frac{1}{\prod_{i=0}^{m-1}(e_i!) (e_m-1)!}
 (d-m)(d-m-1)\dots(d-m-|\lambda'|+1)
 s_{|\widetilde{\lambda'}|} + \dots \]
and that $s_{|\widetilde{\lambda'}|} \partial(b^m)=s_{|\widetilde{\lambda'}|}
s_{m-1}=s_{|\tilde{\lambda}|}$ is the only Schur polynomial not divisible by $ab$, we get
 the result.
\end{proof}

Note in particular, that the $d$-degree of $\Pl_{\lambda,|\tilde{\lambda}|}(d)$ reaches
$|\lambda|$, the highest possible by Theorem \ref{deg=la}.
The idea of writing $\left[ \,\overline{Y}_{\lambda'}(d') \right]_{m/d'}$ as in
(\ref{eq_Ylambdasubstituted_as_sum}) can be carried further to calculate the
exact $d$-degrees of all Pl\"ucker numbers:

\begin{theorem} \label{pl-la-drop2}
  Let $\lambda_1$ be the largest number in the partition $\lambda$. Then
  \[ \twocase{\deg\big(\Pl_{\lambda;|\tilde{\lambda}|-2j}(d)\big)=}
    {|\lambda|}{j\leq |\tilde{\lambda}|-\lambda_1+1,}
  {|\lambda|-(j- |\tilde{\lambda}|+\lambda_1-1)}{j> |\tilde{\lambda}|-\lambda_1+1.} \]
\end{theorem}
In other words, $\Pl_{\lambda;|\tilde{\lambda}|-2j}(d)$ has degree $|\lambda|$ for $j=0,\dots,|\tilde{\lambda}|-\lambda_1+1$, then by increasing $j$ by one, the degree
drops by one.

The proof---which is straightforward but laborious---can be found in \cite{juhasz2024leadingtermsofplucker}.
\begin{example}
For $\lambda=(10,2,2)$ we have $|\lambda|=14$, $|\tilde{\lambda}|=11$, $\lambda_1=10$
and $|\tilde{\lambda}|-\lambda_1+1=2$, implying
\[ \deg\big(\Pl_{10,2,2;11}(d)\big)=\deg\big(\Pl_{10,2,2;9}(d)\big)=\deg\big(\Pl_{10,2,2;7}(d)\big)=14, \]
and
\[ \deg\big(\Pl_{10,2,2;5}(d)\big)=13,\ \ \deg\big(\Pl_{10,2,2;3}(d)\big)=12,\ \ \deg\big(\Pl_{10,2,2;1}(d)\big)=11.\]
\end{example}

If $\lambda_1$ is not much bigger than the other $\lambda_i$---exactly
if $\lambda_1\leq \lceil|\tilde{\lambda}|/2\rceil +1$---then all Pl\"ucker numbers have degree
$|\lambda|$. We saw this in Example \ref{pl22} and \ref{pl22;2} for the bitangents:
both $\Pl_{2,2;0}(d)$ and $\Pl_{2,2;2}(d)$ have degree $|\lambda|=4$.
A slightly bigger example is $\lambda=(4,3,2)$, where all Pl\"ucker numbers have degree $|\lambda|=9$.

\section{The class of \texorpdfstring{$m$}{m}-flexes} \label{sec:m-flexes} This is a family of Pl\"ucker problems where closed formulas can be given. The formula for the equivariant classes $\left[\,\overline  Y_m(d)\right]$ were already computed in \cite{kirwan} and, more explicitely, in \cite[Ex. 3.7 (4)]{fnr-root}. However, deduction of the corresponding Pl\"ucker numbers given below is new.

Using \eqref{Y_m}, we can write $\left[\,\overline Y_m(d)\right]=
\partial \left( \prod_{i=0}^{m-1}\big( ia+(d-i)b \big)  \right)$ in the Schur polynomial basis:

\begin{equation*}\label{eq:m-flex}
  \begin{split}
    \prod_{i=0}^{m-1}\big( ia+(d-i)b \big) =&\prod_{i=0}^{m-1}\big(i(a-b)+db\big)=(db)^m \prod_{i=0}^{m-1} \left( 1+i\frac{a-b}{db} \right) \\
                            =&\sum_{k=0}^{m-1}\sigma_k(1,2,\dots,m-1)(a-b)^k(db)^{m-k},
  \end{split}
\end{equation*}
where $\sigma_k$ denotes the $k$-th elementary symmetric polynomial. By definition, we have a connection with the Stirling numbers of the first kind:
\[ \sigma_k(1,2,\dots,m-1)=\stir{m}{m-k}.  \]

Almost by definition, we have
\[ \partial \left( a^ib^{m-i} \right) = s_{m-i-1,i},\]
where the Schur polynomials in the right-hand side are indexed by vectors of integers. Using the straightening law, we can restrict ourselves to Schur polynomials indexed by partitions:

\begin{equation*} \label{divdiff-schur}
 \twocase{\partial\left( a^ib^{m-i}\right)=} {s_{m-i-1,i}}{2i<m,}{-s_{i-1,m-i}}{2i>m}
\end{equation*}
and $\partial \left(a^ib^i\right)=0$.

This implies that
\begin{theorem}\label{m-flex-coeffs}
For $m\geq 2i+1$ the coefficient of $d^{m-k}s_{m-i-1,i}$ in $\left[\,\overline Y_m(d)\right]$ is
\begin{multline*}
\coef\!\big(d^{m-k}s_{m-i-1,i},\left[\,\overline Y_m(d)\right]\big)=\\
\begin{cases}
  (-1)^{k+i}\dbinom{k}{i}\stir{m}{m-k}  & \text{if }i\leq k<m-i, \\
  \left( (-1)^{k+i}\dbinom{k}{i}- (-1)^{k+m-i} \dbinom{k}{m-i}\right)\stir{m}{m-k} & \text{if }
m-i\leq k < m.
\end{cases}
\end{multline*}

\end{theorem}

Note that adding up the above coefficients for $i=0$,
\[ \Pl_{m;m-1}=\coef(s_{m-1},\left[ \,\overline{Y}_m(d) \right])=\sum_{k=0}^{m-1}(-1)^k \stir{m}{m-k}
d^{m-k}=d(d-1)\dots(d-m+1), \]
we get back Theorem \ref{plpoint} in the $\lambda=\left( m \right)$ special case.
In other words,
\begin{proposition}
  For a generic degree $d$ hypersurface in $\P(\C^n)$, the number of
  $(n-1)$-flex lines through a generic point of $\P(\C^n)$ is $d(d-1)\cdots(d-n+2)$.
\end{proposition}

 \subsection{Enumerative consequences}
 Specializing Theorem \ref{m-flex-coeffs} to $m=2i+1$, we obtain
\[\coef\!\big(d^{m-k}s_{i,i},\left[\,\overline Y_m(d)\right]\big)=(-1)^{k+i}\binom{k+1}{i+1}\stir {2i+1}{2i+1-k}\]
for $k=i,i+1,\dots,m-1$,
using the  identity $\binom{k}{i}+\binom{k}{i+1}=\binom{k+1}{i+1}$. In other words,
\begin{proposition} \label{plm} The number of $m=2n-3$-flexes of a degree $d$ hypersurface in $\P(\C^n)$ is
\[\Pl_m(d)=\sum_{u=1}^{n-1}(-1)^{u+n+1}\stir mu\binom{m-u+1}{n-1}d^u.  \]
\end{proposition}
In particular, for $d=m$ we obtain
\begin{theorem} \label{hyperflex-number}
A generic degree $d=2n-3$ hypersurface in $\P(\C^n)$ possesses
  \[ \sum_{u=1}^{n-1}(-1)^{u+n+1}\stir du\binom{d-u+1}{n-1}d^u\]
  lines which intersect the hypersurface in a single point.
 \end{theorem}
 For $n=3$ it says that a generic cubic plane curve has 9 flexes. For $n=4$ we obtain the classical result that a generic quintic surface has 575 lines which intersect the surface in a single point (see e.g. \cite[Thm.~11.1]{eh3264}). For $n=5,6, 7, 8,9$ we get
 \[                              99715,\ 33899229,\ 19134579541,\ 16213602794675,\ 19275975908850375. \]

 \subsection{Lines on a hypersurface}\label{subsec_linesonhypersurface}

 After discovering Theorem \ref{hyperflex-number}, we found a formula of Don Zagier
 \cite[p.~26]{grunberg-moree-zagier} on the classical problem of counting lines on hypersurfaces. Comparing his formula with Theorem \ref{hyperflex-number} shows a surprisingly simple connection between the two problems. In this section we prove this connection directly and we also generalize it. As a byproduct, we obtain a new proof of Zagier's result.

The Fano variety $F_f$ of lines on a degree $d$  hypersurface $Z_f\subset \P(\C^{n+1})$ is the zero locus of the section $\sigma_f:\gr_2(\C^{n+1})\to \Pol^d(S)$. Therefore, for a generic $f$ we have
  \[   \big[F_f\subset \Gr_2 (\C^{n+1})\big]=e\big(\Pol^d(S) \big).  \]
  For $d=2n-3$ we have finitely many lines, and to find their number we need to calculate the coefficient of $s_{n-1,n-1}$ in $e\big(\Pol^d(S)\big)=\prod_{i=0}^{d}\big( ia+(d-i)b \big) $.

To establish the promised connection we need the following
\begin{proposition}\label{e-yd} Written in Chern roots we have
  \[   e\!\left( \Pol^d(\C^2) \right)=dab \left[\,\overline{Y}_d(d) \right].   \]
\end{proposition}
\begin{proof}
  By basic properties of the divided difference, we have
  \begin{equation}\label{eq_Ymfactorized}
  \left[ \, \overline{Y}_d(d) \right]=
  \partial\left(\prod_{i=0}^{d-1}\big( ia+(d-i)b \big) \right)=\prod_{i=1}^{d-1} \big(ia+(d-i)b\big)\partial(db)=d\prod_{i=1}^{d-1}\big( ia+(d-i)b \big) ,
\end{equation}
implying that
  \[   e\!\left( \Pol^d(\C^2) \right)=dab \left[\,\overline{Y}_d(d) \right].   \]
\end{proof}
In retrospect, this connection could have been known to the authors: The first appearance of the factorized
form (\ref{eq_Ymfactorized}) may be in \cite{balazs-tezis}. It can also be obtained by the classical
resolution method.

\begin{observation} For Schur polynomials in variables $a,b$ we have
  \[ s_{i+1,j+1}=abs_{i,j}  \]
  for all $i\geq j$.
\end{observation}

This observation together with Proposition \ref{e-yd} immediately implies
\begin{corollary}
Expressing  $e\!\left(\Pol^d(\C^2)\right)$ and $\left[\,\overline Y_d(d)\right]$ in Schur basis:
\[e\!\left( \Pol^d(\C^2) \right)=\sum_{j=0}^{\lfloor(d+1)/2\rfloor}u_j s_{d+1-j,j} \ \text{ and } \
\left[\,\overline{Y}_d(d)\right]=\sum_{j=0}^{\lfloor(d-1)/2\rfloor}v_j s_{d-1-j,j},\]
we have the identities $u_{j+1}=dv_j$ for $j=0,\dots,\lfloor(d-1)/2\rfloor$ and $u_{0}=0$.
\end{corollary}

If $d=2n-3$, the case $u_{n-1}=dv_{n-2}$ implies

\begin{theorem} \label{lines-on-vs-hyperflexes}
The number of lines on a generic degree $d=2n-3$ hypersurface in $\P(\C^{n+1})$ is $d$ times the number of hyperflexes of a generic degree $d$ hypersurface in $\P(\C^{n})$.
\end{theorem}

\begin{remark}If would be nice to have a geometric explanation of this connection. Igor Dolgachev recommended to use cyclic coverings: let $f(x_1,\dots,x_n)\in \pol^d(\C^n)$ be generic and
\[ \tilde f(x_1,\dots,x_n,x_{n+1}):=f(x_1,\dots,x_n)+x_{n+1}^d\]
with $d=2n-3$. Then the projection $\pi:Z_{\tilde f} \to \P(\C^{n})$ has the following property: The preimage of a hyperflex to $Z_f$ is the union of $d$ lines on $Z_{\tilde f}$. The case $n=3$ is explained in \cite[Ex 9.1.24]{dolgachev}. It remains to be shown that for generic $f$ the section $\sigma_{\tilde f} :\gr_2(\C^{n+1})\to \Pol^d(S)$ is transversal to the zero section and have no other zeroes. We will not pursue this approach further in this paper.
\end{remark}

\begin{corollary}[Zagier's formula] The number of lines on a generic degree $d=2n-3$ hypersurface in $\P(\C^{n+1})$ is
 \[ \sum_{u=1}^{n-1}(-1)^{u+n+1}\stir du\binom{d-u+1}{n-1}d^{u+1}.\]
\end{corollary}

The identities $u_{j+1}=dv_j$ imply the following generalization of Theorem \ref{lines-on-vs-hyperflexes}:

\begin{theorem} \label{gen-lines-on-vs-hyperflexes}
  Let $d$ be any degree and
  choose $n$ and $0 \leq i \leq d-1$ such that $d-1+i=2(n-2)$. Then the number of lines on a generic degree
  $d$ hypersurface in $\P(\C^{n+1})$ intersecting a generic $(i+1)$-codimensional projective subspace is
  $d$ times the number of hyperflexes of a generic degree $d$ hypersurface in $\P(\C^{n})$ intersecting a
  generic $(i+1)$-codimensional projective subspace.

\end{theorem}

\section{Asymptotic behaviour of the Pl\"ucker number \texorpdfstring{$\Pl_{\lambda;i}(d)$}{Pllambda;i(d)}}
\label{sec:poly-of-Pl} Theorem \ref{deg=la} implies that the  Pl\"ucker numbers $\Pl_{\lambda;i}(d)$ are polynomials in $d$ and have degree at most $|\lambda|$. In this section we calculate the coefficient of $d^{|\lambda|}$ in $\Pl_{\lambda;i}(d)$ by relating it to certain Kostka numbers.

This coefficient informs us about the asymptotic behaviour of the  Pl\"ucker number $\Pl_{\lambda;i}(d)$ as $d$ tends to infinity, so we will call it the \emph{asymptotic  Pl\"ucker number} $\apl_{\lambda;i}$ (or $\apl_{\lambda}$ for $i=0$).
Theorem \ref{pl-la-drop2} shows that for some $i$ the polynomial $\Pl_{\lambda;i}(d)$
can have degree less than $|\lambda|$. In these cases $\apl_{\lambda;i}=0$.
 For example, the number of flexes is $\Pl_3(d)=3d(d-2)$, so the coefficient of $d^3$,
$\apl_{3}$ is zero. More generally, Proposition \ref{plm} shows  that
$\deg_d(\Pl_m(d))=(m+1)/2$, so the degree of $\Pl_\lambda(d)$ can be much lower than $|\lambda|$.

Recall that Kostka numbers can be defined as coefficients of the Schur expansion of the complete symmetric polynomials:
\[ h_\nu=\sum K_{\mu,\nu}s_\mu. \]

Then the leading term formula of Theorem \ref{fundclass-leading} immediately implies

\begin{theorem}\label{thm:apl}
Let $\lambda=(2^{e_2},\dots,m^{e_m})$ be a partition without $1$'s and $j\leq \lfloor |\tilde{\lambda}|/2 \rfloor$ a nonnegative
 integer. Then
 \[ \apl_{\lambda;|\tilde{\lambda}|-2j}= \frac{K_{(|\tilde{\lambda}|-j,j),\tilde{\lambda}}}{\prod_{i=2}^m \left( e_i!\right)}. \]
\end{theorem}

In particular, from basic properties of the Kostka numbers (see e.g. \cite[Exercise~A.11.]{fulton-harris}) we obtain that
\begin{corollary}\label{apl=0} The asymptotic  Pl\"ucker number $\apl_\lambda$ is zero if and only if $\lambda_1\geq |\tilde{\lambda}|/2+2$.
\end{corollary}
Note that this can also be easily deduced from Theorem \ref{pl-la-drop2}.

\begin{remark}Kostka numbers have an interpretation as  solutions to linear Schubert problems:
Let $n=|\mu|-j+2$, then the Kostka number $K_{(n-2,j),\mu}$ is the number of lines in
$\mathbb{P}(\mathbb{C}^n)$ intersecting generic subspaces of codimension
$\mu_1+1,\dots,\mu_k+1$ and $n-j-1$.

For example, the number of 4-tangents is
\[ \Pl_{2^4}(d)=\frac1{12}d (d-7)(d-6)(d-5)(d-4)(d^3 + 6 d^2 + 7 d - 30),\]
therefore $\apl_{2^4}=\frac1{12}=\frac2{4!}$, and this 2 can be interpreted as
the solution of the famous Schubert problem: how many lines intersect four generic lines in $\P^3$?
\end{remark}

More generally, for $\lambda=2^{2(n-2)}$, we have
\[K_{(n-2)^2,1^{2(n-2)}}= C(n-2),\]
where $C(n)$ denotes the $n$-th Catalan number, e.g. the number of standard
Young tableaux for the 2-by-$n$ rectangle. This implies
\begin{proposition} \label{multi-tangent}
 \[ \apl_{2^{2(n-2)}}=\frac{C(n-2)}{(2n-4)!}=\frac{1}{(n-2)!(n-1)!}. \]
\end{proposition}
Similarly, for $\lambda=3^{n-2}$---as $K_{(n-2)^2,2^{n-2}}=R(n-2)$---we have
\begin{proposition}
 \[ \apl_{3^{n-2}}=\frac{R(n-2)}{(n-2)!}, \]
 where $R(n)$ is the $n$-th Riordan number ($R(3)=1,\ R(4)=3,\ R(5)=6,\ R(6)=15$, see OEIS \url{https://oeis.org/A005043}).
\end{proposition}

On the other hand, for $\lambda=(n-1,n-1)$---as $K_{(n-2)^2,(n-2)^2}=1$---we have
\begin{proposition}
 \[ \apl_{n-1,n-1}=\frac12.\]
\end{proposition}

\bigskip

\section{Comparison with the classical non-equivariant method} \label{sec:compare}
The goal of this section is to build a bridge between the classical method of solving enumerative problems and the equivariant one.

\subsection{The general setup}
The classical method for computing the cohomology class of a closed subvariety
$Z \subset X$ is to give  a resolution $\varphi: \tilde{Z} \to X$ of $Z$ and compute
$\varphi_! 1$.
For a general $\,\overline{\mathcal{T}_\lambda Z_f} \subset \Gr_2 (\mathbb{C}^n)$,
to find a resolution for which we can calculate this pushforward is difficult.

Instead, we use equivariant cohomology classes that can be computed using equivariant
methods, such as localization and the ABBV integral formula. Then we write $Z \subset X$
as a locus of a sufficiently transversal section $\sigma$ in
\begin{equation*}
\begin{tikzcd}[column sep=0.1in]
P \times_G Y \ar[r,phantom,"\subseteq",description] &A=P \times_G V \ar[d] \\
Z=\sigma^{-1}\left( P \times_G Y \right) \ar[r,phantom,"\subseteq",description] &
X \ar[u, bend right=30,"\sigma" right]
\end{tikzcd}
\end{equation*}
so we can use the universal property of the equivariant class
$\left[ Y \subset V \right]_G$
to compute $\left[ Z \subset X \right]$.

However, if
\begin{equation}\label{eq_genfiberedres}
\begin{tikzcd}
E \ar[rr,bend left=20, "\varphi"] \ar[dr] \ar[r,hook, "j" below] &
M \times V \ar[r,"\pi" below] \ar[d]
& V \\
 & M &
\end{tikzcd}
\end{equation}
is an equivariant fibered resolution of $Y \subset V$,
then (by Lemma \ref{push-from-subbundle}) not only
$ \left[ Y \subset V \right]_G={\pi}_! e_G\left( V/E \right)$ but also
we can avoid using equivariant theory as from
(\ref{eq_genfiberedres}) we can derive a resolution of $Z \subset X$ as follows.
All the maps in diagram (\ref{eq_genfiberedres}) are $G$-equivariant, so we can associate
it to the principal $G$-bundle $P \to X$. Complete the resulting diagram with
the associated bundle $p: \mathbf{M}=P \times_G M \to X$ to get
\begin{equation*}
\begin{tikzcd}
  \mathbf{E}=P \times_G E \ar[r,hook,"\boldsymbol{j}"] \ar[dr] &
\mathbf{V}=P \times_G \left( M \times V \right)=p^* A \ar[r,"\boldsymbol\pi"] \ar[d] &
A=P \times_G V \ar[d] \\
& \mathbf{M}=P \times_G M \ar[r,"p"]
\ar[u, bend right=30,"\bar{\sigma}=p^*\sigma" right] &
X \ar[u, bend right=30,"\sigma" right]
\end{tikzcd}.
\end{equation*}
It follows from the construction that $\bar{\sigma}$ is transversal to
$\mathbf{E} $, hence $p$ restricted to
$\bar{\sigma}^{-1}\left( \mathbf{E} \right)$ gives a
resolution of $Z=\sigma^{-1}\left( P \times_G Y \right)$.

We get that
\[ \left[ Z \subset X \right]=p_! e\left(\mathbf{V}/\mathbf{E}\right). \]
Note that, to keep our formulas shorter,
we use the same notation for a bundle and its injective image.

\subsection{The case of \texorpdfstring{$m$}{m}-flex lines}
Our recursive method provides an equivariant fibered resolution  for
$\,\overline{Y}_\lambda \subset \Pol^{d}(\mathbb{C}^2)$ when $\lambda=(m)$.
In what follows, we work out
the details of the above general method for this case.
For some $m$'s calculations are described in \cite[Ch.~11]{eh3264}.

We recall the construction from Section \ref{sec:variety-of-tangent}: $f\in\Pol^d(\C^n)$ defines a hypersurface $Z_f \subset \P(\C^n)$. It also induces a section
\[\sigma_f:X=\gr_2(\C^n)\to A=\Pol^d(S),\, \sigma_f(W):=f|_W,\]
where $S\to \gr_2(\C^n)$ is the tautological bundle. We identify the variety of tangent lines of type $\lambda$ as
\[ {\mathcal{T}}_\lambda Z_f=\sigma_f^{-1}({Y}_\lambda(d)).\]
Once we have $\left[\,\overline{Y}_\lambda(d)\right]$, we can get $\left[\,\overline{\mathcal{T}_\lambda Z_f}
\subset \gr_2(\C^n)\right]$ by substituting $c_i \mapsto c_i(S^\vee)$ as in Section
\ref{subsec:calccohomclass}.

The covering map, described in Section \ref{subsec:twistingwithlinebundle}, becomes a fibered resolution for $\lambda=(m)$:
\begin{equation}\label{eq:Ymfiberedres}
\begin{tikzcd}
E=\Pol^{d'}(\C^2)\otimes \Pol^m(\C^2/\gamma) \ar[rr,bend left=20, "\varphi"] \ar[dr] \ar[r,hook, "j" below] &
\P(\C^2) \times \Pol^{d}(\mathbb{C}^2) \ar[r,"\pi" below] \ar[d]
& \Pol^{d}(\mathbb{C}^2) \\
 & \P(\C^2) &
\end{tikzcd}.
\end{equation}

Associate diagram (\ref{eq:Ymfiberedres}) to the frame bundle $P=\Inj(\mathbb{C}^2, S) \to X=\gr_2(\C^n)$ to get:
\begin{equation}\label{eq_TmZfresolution}
\begin{tikzcd}
  \mathbf{E}=\Pol^{d'}(S^2) \otimes \Pol^{m}(S^2/S^1) \ar[r,hook,"\boldsymbol{j}"] \ar[dr] &
\mathbf{V}=\Pol^{d}(S^2)=p^*  \Pol^{d}(S)\ar[r,"\boldsymbol\pi"] \ar[d] &
A=\Pol^{d}(S) \ar[d] \\
& \mathbf{M}=\fl_{1,2}\left( \C^n \right) \ar[r,"p"]
\ar[u, bend right=30,"\bar{\sigma}_f=p^*\sigma_f" right] &
\gr_2(\C^n) \ar[u, bend right=30,"\sigma_f" right]
\end{tikzcd},
\end{equation}
where $S^i$ is the tautological bundle of rank $i$ over the flag manifold
 $\fl_{1,2}(\C^n)$ and $d'=d-m$.
(\cite[Ch.~11]{eh3264} uses the notation $\mathbb{G}(1,n-1)$ for the Grassmannian $\gr_2(\C^n)$ and $\Psi$ for the flag manifold $\fl_{1,2}(\C^n)$.)

As we explained in the previous section, if $f$ is generic, then $\left[\,\overline{\mathcal{T}_m Z_f}\subset
\Gr_2(\mathbb{C}^n)\right]=p_!e(\mathbf{V}/\mathbf{E})$, where $p:\fl_{1,2}(\C^n)\to \gr_2(\C^n)$ denotes the
projection. To calculate the pushforward we use
\begin{proposition}\label{projective-pushforward}  \cite[Prop. 10.3]{eh3264}
  Let $D\to X$ be a rank $k$ bundle over a smooth $X$, and $p:\P (D)\to X$ its projectivization. Then all $\alpha\in H^*(\P (D))$ is of the form $\alpha=\sum \beta^i p^*m_i$, where $\beta=c_1(\gamma^\vee\to \P (D))$ with $\gamma \to \P(D)$ the tautological line
bundle of $\P(D)$
and
\[ p_!\alpha=\sum s_{i-k+1} m_i,\]
where $1/c(D)=s(D)=1+s_1+\cdots$ is the Segre class of $D$
(or, equivalently, $s_i=s_i(c_j(D^\vee))$, the Schur polynomial in the Chern classes of the dual bundle
$D^\vee$).
\end{proposition}
Namely, recall that $\fl_{1,2}(\C^n)$ is the total space of the projective bundle $p:\P(S)\to \gr_2(\C^n)$ and
$S^1\to \fl_{1,2}(\C^n)$ is the tautological line bundle of $\P(S)$.

To stay close to the notation of \cite[Ch.~11]{eh3264}, we choose generators of
$H^*(\fl_{1,2}(\C^n))$
\[ c_1((S^2)^\vee)=\sigma_1,\   \    c_1((S^1)^\vee)=\zeta:\]
 $S^1$ is a subbundle of $S^2$ so $S^2/S^1$ is also a line bundle over $\fl_{1,2}(\C^n)$ with $c_1((S^2/S^1)^\vee)=\eta=\sigma_1-\zeta$.
Using these generators, we can apply Proposition \ref{projective-pushforward} and
calculate the pushforward $p_!$ as
\begin{equation*}\label{p!}
  p_!(\zeta^a\sigma_1^b)=s_{a-1}\sigma_1^b, \text{ where } s=1+\sigma_1+\sigma_1^2-
\sigma_2+\dots
\end{equation*}
is the Segre class of $S$.
Note that we don't distinguish between cohomology classes and their pullbacks, and hence
in the remainder of this Section we denote by $\sigma_i$ the $i$-th Chern class of $S^\vee \to \gr_2 \C^n$.
Be aware that in the earlier Sections we used $c_i$ for $c_i(S^\vee)$.

 The Chern classes of the \quot{bold} bundles can be obtained by substituting the corresponding Chern roots:
\begin{equation*}\label{cVV}
  c(\mathbf{V})=\prod_{i=0}^{d} \big( 1+(d-i)\zeta+i\eta \big) =\prod_{i=0}^{d} \big( 1+(d-2i)\zeta+i\sigma_1\big).
\end{equation*}
Similarly,
\begin{equation*}\label{cEE}
  c(\mathbf{E})=\prod_{i=0}^{d-m} \big( 1+m(\sigma_1-\zeta)+(d-m-i)\zeta+i(\sigma_1-\zeta)\big),
\end{equation*}
implying that
\begin{equation*}\label{cV/E}
  c(\mathbf{V}/\mathbf{E})=\prod_{i=0}^{m-1} \big( 1+(d-i)\zeta+i \eta \big) =\prod_{i=0}^{m-1} \big(
  1+(d-2i)\zeta+i\sigma_1 \big)
\end{equation*}
and
\begin{equation*}\label{eV/E}
  e(\mathbf{V}/\mathbf{E})=c_m(\mathbf{V}/\mathbf{E})=\prod_{i=0}^{m-1} \big( (d-i)\zeta+i \eta \big) =
  \prod_{i=0}^{m-1} \big( (d-2i)\zeta+i\sigma_1 \big).
\end{equation*}

The case $m=2$,
\begin{equation*}
  p_!c_2(\mathbf{V}/\mathbf{E})=p_!\big(d\zeta((d-2)\zeta+\sigma_1)\big)=d(d-1)\sigma_1
=d(d-1)s_1
\end{equation*}
is equivalent to the Pl\"ucker formula $\Pl_{2;1}(d)$ of Example \ref{pl2;1}.
\bigskip

The case $m=3$ is
\begin{equation*}
 \begin{aligned}
  p_!c_3(\mathbf{V}/\mathbf{E})= &\, p_!\big(d\zeta((d-2)\zeta+\sigma_1)((d-4)\zeta+2\sigma_1)\big) \\
= &\, p_!\big(2d\zeta\sigma_1^2+d(3d-8)\zeta^2\sigma_1 + d(d-2)(d-4)\zeta^3 \big) \\
=&
d(d-1)(d-2)\sigma_1^2-d(d-2)(d-4)\sigma_2
=d(d-1)(d-2)s_2 + 3d(d-2)s_{1,1},
 \end{aligned}
\end{equation*}
where the coefficients are the  Pl\"ucker numbers $\Pl_{3;2}(d)$ and $\Pl_3(d)$ of Example \ref{pl22}.
\bigskip

For a specific $n$ calculations can be simplified by the observation that $\zeta^n=0$ since $S^1$ is the pullback of the tautological bundle of $\P(\C^n)$.

\begin{remark}
  A different formula can be given for the pushforward map in the special case of $\P^1$-bundles
  $p:\mathbb{P}(D) \to X$ like $\fl_{1,2}(\C^n)\to \gr_2(\C^n)$.
  Let $\zeta,\eta$ denote the Chern roots of the rank two bundle $D^\vee \to X$:
  $\zeta:=c_1(\gamma^\vee\to \P(D))$  and  $\eta:=c_1((p^*D/\gamma)^\vee\to \P(D))$.
  Then for any polynomial $q(\zeta,\eta) \in H^*(\mathbb{P}(D))$ its pushforward along $p:\P(D)\to X$ is given by
  \[ p_!q(\zeta,\eta)=\partial q.\]
  This is easy to prove directly, but also follows from the equivariant pushforward formula \eqref{equiv-integral}. The calculations above become simpler if we use the variables $\zeta,\eta$ and this pushforward formula.
\end{remark}

\subsection{Incidence varieties}
It is instructive to describe
the section $\bar{\sigma}_f$ and the subvariety $\bar{\sigma}_f^{-1}(\mathbf{E}) \subset \fl_{1,2}(\C^n)$
of \eqref{eq_TmZfresolution}.
\[ \bar{\sigma}_f\left( (W,L) \right)=f|_W \in \mathbf{V}_{(W,L)}=\Pol^{d}(W), \]
hence
\[ \bar{\sigma}_f\left( (W,L) \right) \in \mathbf{E}_{(W,L)}=
 \Pol^{d'}(W) \otimes \Pol^{m}(W/L) \]
is equivalent to having a basis
$x,y$ of $W^\vee$ such that $x(L)=0$ and $f|_W=x^mp(x,y)$ for some polynomial $p$ of degree $d'=d-m$.
Therefore
\[ I_m:=\bar{\sigma}_f^{-1}(\mathbf{E})=\{(W,L): [W] \text{ and } Z_f \text{ has a point of contact of order at least $m$ at } [L] \} ,\]
which is the usual resolution of $\,\overline{\mathcal{T}_{m} Z_f}$.
We will call it the \emph{incidence variety}.
\bigskip

Notice that the bundle $\mathcal{E}$ used in \cite[Ch.~11]{eh3264} looks different than our $\mathbf{V}/\mathbf{E}$. They have the same Chern classes, so the calculations are the same. They are probably also isomorphic.

\begin{remark}\label{incidence} The construction of the incidence variety can be generalized.
Let $m$ be an element of $\lambda$, not necessarily equal to $\lambda_1$. Denote
by $\lambda'$ the partition $\lambda$ minus $m$. Similarly to what we had for the
covering map constructed in Section \ref{subsec:twistingwithlinebundle}, $\mathbf{E}$ has a subbundle
\[ Y_{\lambda'}\left( \mathbf{E} \right)=Y_{\lambda'}\left(
 \Pol^{d'}(S^2) \otimes \Pol^{m}(S^2/S^1)\right) \]
corresponding to the invariant subvariety $Y_{\lambda'}(d') \subset
\Pol^{d'}(\mathbb{C}^2)$.

Choose a generic $f\in \pol^{d}(\C^n)$. It induces a section
$\bar\sigma_f: \fl_{1,2}(\C^n)\to \mathbf{V}=\pol^{d}(S^2)$. Then
 $\bar\sigma^{-1}_f\big(\,\overline Y_{\lambda'}(\mathbf{E})\big)$ can be identified with the \emph{incidence variety $I_{\lambda';m}$ of $m$-flex points and $\lambda$-lines for $f$}.
Therefore, by Lemma \ref{push-from-subbundle} and Corollary \ref{twist4bundle},
\[ [I_{\lambda';m}\subset \fl_{1,2}]=e(\mathbf{V}/\mathbf{E})\cdot \left[\,\overline Y_{\lambda'}(\mathbf{E})\subset \mathbf{E}\right]\]
and
\[ \left[\,\overline Y_{\lambda'}(\mathbf{E})\subset \mathbf{E}\right]=\left[\,\overline Y_{\lambda'}(d')\right]\big(\eta+\frac{m}{d'}\eta,\zeta+\frac{m}{d'}\eta\big),\]
where we substitute into the equivariant class
$\left[ \,\overline{Y}_{\lambda'}(d') \right]$ expressed in Chern roots $a,b$,
see Section \ref{sec:conventions}.

Since $p|_{I_{\lambda';m}}$ is an $e_m$-to-1 branched covering of $\,\overline{\mathcal{T}_{\lambda} Z_f}$, we can calculate its cohomology class:
\begin{equation*}\label{nonequi-recursion}
  \left[\,\overline{\mathcal{T}_{\lambda} Z_f} \subset \Gr_2(\mathbb{C}^n)\right]=\frac{1}{e_m}p_!\Big(e(\mathbf{V}/\mathbf{E})\cdot \left[\,\overline Y_{\lambda'}(\mathbf{E})\subset \mathbf{E}\right]\Big).
\end{equation*}
This is the non-equivariant version of the proof of Theorem \ref{recursion4Y}. Notice that $I_{\lambda';m}$ is not smooth in general.
\end{remark}

\begin{example}
For $\lambda=\left( 2,2 \right)$ the class of the incidence variety $I_{2;2}$
of bitangent points and bitangents is
\begin{equation*}
 \begin{aligned}
\left[ I_{2;2} \subset \fl_{1,2}(\C^n) \right]=& e\left( \mathbf{V}/\mathbf{E} \right)
\left[ \,\overline{Y}_{2}\left( \mathbf{E} \right) \subset \mathbf{E} \right]\\=&
\prod_{i=0}^{1} \big( (d-i)\zeta+i \eta \big)  \cdot \left[ \,\overline{Y}_2 (d-2)
 \right]\!\left( \eta+ \frac{2}{d-2}\eta, \zeta+ \frac{2}{d-2}\eta \right)\\
 =&d(d-3)(d+2)\zeta \sigma_1^2+d(d-3)(d^2-8)\zeta^2 \sigma_1-4d(d-2)(d-3)\zeta^3.
\end{aligned}
\end{equation*}
As $1/2$-times its pushforward along $p$ we get
\begin{equation*}
\begin{aligned}
\left[ \,\overline{\mathcal{T}_{2,2}Z_f}\subset \gr_2(\C^n) \right]=&\frac{1}{2} p_!
\left[ I_{2;2} \subset \fl_{1,2}  \right]= \frac{1}{2}
d(d-1)(d-2)(d-3)\sigma_1^2+2d(d-2)(d-3)\sigma_2\\=&\frac{1}{2}d(d-1)(d-2)(d-3)s_2
+\frac{1}{2}d(d-2)(d-3)(d+3)s_{1,1},
\end{aligned}
\end{equation*}
agreeing with \eqref{Y22}.
\end{example}
\section{Further enumerative problems} \label{sec:further}
\subsection{The universal hypersurface and Pl\"ucker numbers for linear systems} \label{sec:correspondence}
A more general construction considers all hypersurfaces $Z_f \subset \P(\C^n)$ at once. Consider the vector bundle
\begin{equation*}
\begin{tikzcd}
A_u:=\Hom\big(L,\pol^d(S)\big) \ar[r] &\P\big(\Pol^d(\C^n)\big)\times \gr_2(\C^n),
\end{tikzcd}
\end{equation*}
where $L$ and $S$ are the tautological bundles over $\P\big(\Pol^d(\C^n)\big)$
and $\gr_2\left( \C^n \right)$.

$A_u$ has a section
\[\sigma([f],W)(f): f \mapsto f|_W,\]
the \emph{universal section}.
Applying the construction in Section \ref{subsec:twistingwithlinebundle} to
$A_u \cong \Pol^{d}(S) \otimes L^\vee$ and  $\GL(2)$-invariant subsets
$Y_{\lambda}(d) \subset \Pol^{d}(\mathbb{C}^2)$, we get subbundles $Y_{\lambda}(A_u)$.
The universal section is transversal to the subvarieties $\,\overline{Y}_\lambda(A_u)$ (see Example \ref{ex:transversality_univsection}).
The cohomology classes $\left[ \sigma^{-1}(\,\overline{Y}_\lambda(A_u)) \right]$
are sources of answers for new enumerative problems. They can be calculated
using Corollary \ref{twist4bundle} from the equivariant classes
$\left[ \,\overline{Y}_\lambda (d) \right]$ expressed in Chern classes $c_1,c_2$ (see Section
\ref{sec:conventions}) by substituting
\[c_1\mapsto c_1 +\frac2d\xi,\ c_2\mapsto c_2+\frac1d\xi c_1+\frac 1{d^2}\xi^2,\]
where on the right-hand side of these substitutions $c_i$ and $\xi$ denote
the Chern classes of the duals of $S$ and $L$ respectively.

For example,
\[ \left[\sigma^{-1}(\,\overline{Y}_2\left( A_u \right))\right]=d(d-1)c_1+2(d-1)\xi.\]
Therefore $2d-2$, the coefficient of $\xi$, is the number of degree $d$ curves in a pencil tangent to a given line.

Similarly,
\begin{equation*}\label{y3twisted}
\begin{split}
  [\sigma^{-1}(\,\overline{Y}_3(A_u))] =&  d(d-1)(d-2)c_1^2-d(d-2)(d-4)c_2+3d(d-2)\xi c_1+3(d-2)\xi^2  \\
                           =&  d(d-2)(d-1)s_2+3d(d-2)s_{1,1}+3d(d-2)\xi s_1+3(d-2)\xi^2.
\end{split}
\end{equation*}
The coefficient of $\xi^2$---the degree of the variety $\mathbb{P}\big(\overline{Y}_3(d)\big)$
(see \cite[Cor.~6.4]{fnr-forms})---was already calculated by Hilbert (for all $\,\overline{Y}_\lambda(d)$).
 The only new information is the coefficient of $\xi s_1$,
which for $n=3$ is the number of lines
that go through a point and are flex lines to a member of a pencil of degree $d$ curves.
In other words, $3d(d-2)$ is the degree of the curve in the projective plane of lines that consists of those lines
that are flexes to a member of a given generic pencil.
\subsection{\texorpdfstring{$m$}{m}-flex points of \texorpdfstring{$\lambda$}{lambda}-lines} \label{sec:mflexes-of-lambdalines}
The flag manifold  $\fl_{1,2}(\C^n)$ possesses another fibration $q:\fl_{1,2}(\C^n)\to \P(\C^n)$. We call the $q$-image of the incidence variety $I_{\lambda';m}$ of Remark \ref{incidence} the \emph{ variety of $m$-flex points of $\lambda$-lines}. Since $q:I_{\lambda';m}\to q(I_{\lambda';m})$ is generically one-to-one, we can calculate its cohomology class by pushing forward $[I_{\lambda';m}]$ along $q$.
\begin{example} Let $\lambda=(3,2)$.
Remark \ref{incidence} with $m=3$ and $\lambda'=(2)$ gives that for the incidence
variety $I_{2;3}$ of flex points and $(3,2)$-lines
\begin{equation*}
\begin{aligned}
 \left[ I_{2;3} \subset \fl_{1,2}(\C^n) \right]=&
e\left( \mathbf{V}/\mathbf{E} \right) \left[ \,\overline{Y}_2(\mathbf{E}) \subset
\mathbf{E} \right]\\
=&d\zeta((d-1)\zeta+\eta)((d-2)\zeta+2\eta)\cdot \left[ \,\overline{Y}_2(d-3) \right]\!
\left( \eta+\frac{3}{d-3} \eta, \zeta+\frac{3}{d-3}\eta \right)\\
=&d\zeta((d-1)\zeta+\eta)((d-2)\zeta+2\eta)\cdot(d-3)(d-4)\left( \eta+\frac{3}{d-3} \eta+ \zeta+\frac{3}{d-3}\eta \right).
\end{aligned}
\end{equation*}

The fibration $q:\fl_{1,2}(\C^n) \to \P(\C^n)$ is isomorphic to the projective bundle $\P(\C^n/\gamma) \to \P(\C^n)$, where $\gamma \to \P(\C^n) $ denotes the tautological
bundle. Because of this description, we can use Proposition \ref{projective-pushforward} to calculate the pushforward $q_!$:
  \begin{equation}\label{eq_q!} q_!\left(  \eta^a \zeta^b \right)=s_{a-n+2}\zeta^b, \text{ where } s=1-\zeta \end{equation}
is the Segre class of $C^n/\gamma$ and  $\zeta=c_1(\gamma^\vee \to \P(\C^n))$.

Notice that  $\gamma^\vee \to \P(\C^n/\gamma)$ corresponds to $(S^2/S^1)^\vee \to \fl_{1,2}(\C^n)$, hence the use of $\eta$ in
(\ref{eq_q!}) is consistent with our earlier choice of generators of $H^*(\fl_{1,2}(\C^n))$.

Contrary to the fibration $p: \fl_{1,2}(\C^n) \to \gr_2(\C^n)$, the relative codimension of $q$ depends on $n$. For $n=4$ the nonzero pushfowards are $q_!(\eta^2)=1$ and
$q_!(\eta^3)=-\zeta$, hence for a generic degree $d$ surface, the class of the curve consisting of the 3-flex points of the $(3,2)$-lines is
\[ q_!\left[ I_{2;3} \subset \fl_{1,2}(\C^4) \right]=\left[ q(I_{2;3})\subset \P(\C^4)
 \right]=d(d-4)(3d^2+5d-24) \zeta^2 .\]

Similar calculation shows that for a generic degree $d$ surface the cohomology class of the curve consisting of the 2-tangent points of the $(3,2)$-lines is
  \[ \left[ q\left( I_{3;2} \right) \subset \P(\C^4) \right]=d(d-2)(d-4)(d^2+2d+12)\zeta^2.    \]

\end{example}
\begin{example} The degree of the curve of 4-flex points on a surface is calculated in \cite[p.~399]{eh3264}. The calculation is essentially the same, so we don't repeat it here.
\end{example}
\subsection{\texorpdfstring{$m$}{m}-flex points of \texorpdfstring{$\lambda$}{lambda}-lines for a linear system} \label{sec:combined}
The previous two constructions can be combined without difficulty.

Let $m$ be an element of  $\lambda$, not necessarily equal to $\lambda_1$,
 and consider the vector bundle $\mathbf{V}_u=\Hom(L,\mathbf{V})\to\P \big(\pol^d(\C^n)\big)\times\fl_{1,2}
 (\C^n)$ and its subbundle
\[\mathbf{E}_u=\Hom\left(L,\mathbf{E}\right)=\Hom\big(L,\pol^{d'}(S^2) \otimes \pol^m(S^2/S^1)\big)\to \P\left( \Pol^{d}(\mathbb{C}^n) \right)\times \fl_{1,2}\left( \C^n \right),  \]
where $L$ and $S^i$ are the tautological bundles over $\P(\Pol^{d}(\mathbb{C}^n))$
and $\fl_{1,2}(\C^n)$.

$\mathbf{V}_u$ has a section
\[ \,\overline{\sigma}\left(\left[ f \right],\left( W,L \right)\right): f \mapsto f|_W .\]
Applying the construction in Section \ref{subsec:twistingwithlinebundle} to
$\mathbf{E}_u \cong \Pol^{d'}(S^2) \otimes\Pol^{m}(S^2/S^1) \otimes  L^\vee$ and
$\GL(2)$-invariant subsets $Y_{\lambda'}(d') \subset \Pol^{d'}(\mathbb{C}^2)$, we get subbundles $Y_{\lambda'}(\mathbf{E}_u)$.
The section $\,\overline{\sigma}$ is transversal to these subbundles, hence the class of
their pullbacks $\,\overline{\sigma}^{-1}(\,\overline{Y}_{\lambda'}(\mathbf{E}_u))$
can be calculated using Lemma \ref{push-from-subbundle} and Corollary \ref{twist4bundle} as
\begin{multline*}
\left[\,\overline{\sigma}^{-1}\left( \,\overline{Y}_{\lambda'} \left(\mathbf{E}_u
\right) \right) \subset \P(\Pol^{d}(\mathbb{C}^n)) \times \fl_{1,2}(\C^n) \right]=
 \left[ \,\overline{Y}_{\lambda'}\left (\mathbf{E_u}\right)  \subset \mathbf{V}_u
\right]=\\
e\left( \Hom\left( L, \mathbf{V}/\mathbf{E} \right)  \right)\cdot
\left[ \,\overline{Y}_{\lambda'}(d') \right]\! \left(\eta+\frac{1}{d'}(\xi+m\eta),\zeta+\frac{1}{d'}(\xi+m\eta)\right),
\end{multline*}
where we substitute into the equivariant class
$\left[ \,\overline{Y}_{\lambda'}(d') \right]$ expressed in Chern roots $a,b$,
see Section \ref{sec:conventions}.

The variety $\overline{\sigma}^{-1}\big( \,\overline{Y}_{\lambda'} (\mathbf{E}_u ) \big)$
 can be identified with the \emph{universal incidence variety
of $m$-flex points and $\lambda$-lines} for degree $d$ hypersurfaces in $\P(\C^n)$.

Since the composition
\begin{equation*}
\begin{tikzcd}[column sep=normal]
\P(\Pol^{d}(\mathbb{C}^n)) \times \fl_{1,2}\left( \C^n \right) \ar[r,"\pi_{\fl}"] &
\fl_{1,2}(\C^n) \ar[r,"q"] & \P(\C^n)
\end{tikzcd}
\end{equation*}
restricted to
$\overline{\sigma}^{-1}\left(\, \overline{Y}_{\lambda'} \left(\mathbf{E}_u \right) \right)$
is generically one-to-one, the class of its image is $q_! {\pi_{\fl}}_! \left[
 \,\overline{\sigma}^{-1}\left(\, \overline{Y}_{\lambda'} \left(\mathbf{E}_u \right) \right)
\right]$. This class
provides solutions to further enumerative problems about $m$-tangent points of $\lambda$
-lines in linear systems of hypersurfaces.

\begin{example}
To calculate the degree of the curve consisting of tangent points of bitangent lines in
a pencil of degree $d$ curves, set $\lambda=(2,2)$, $m=2$, $\lambda'=(2)$ and $n=3$.
Then the class of the universal incidence variety is
\begin{multline}\label{eq_22universal}
\left[ \,\overline{\sigma}^{-1}\left(\,\overline{Y}_2(\mathbf{E}_u)\right) \subset
\P(\Pol^{d}(\mathbb{C}^3)) \times \fl_{1,2}(\C^3) \right]=
e\big(\Hom(L,\mathbf{V}/\mathbf{E})\big)\cdot \left[\,\overline Y_{2}(\mathbf{E}_u)\subset \mathbf{E}_u\right]=\\
(d\zeta+\xi)(\eta+(d-1)\zeta+\xi) \cdot (d-3)\big((d-2)(\eta+\zeta)+2\xi+4\eta\big).
\end{multline}
Restricting our attention to a generic pencil of degree $d$ plane curves amounts to
multiplying (\ref{eq_22universal}) by $\xi^{N-1}$ where $N=\dim(\P(\Pol^{d}(\mathbb{C}^3)))$,
while the pushforward along $\pi_{\fl}$ gives the coefficient of the volume form $\xi^N$. Therefore
\[ \left[ \pi_{\fl}\left(\,\overline{\sigma}^{-1}\left(\,\overline{Y}_2(\mathbf{E}_u)\right)
\right)\! \subset \fl_{1,2}(\C^3) \right]=
(d+2)(d-3)\eta^2+2(d-3)(d^2+3d-2)\eta \zeta+(d-3)(4d^2-7d+2)\zeta^2. \]
Finally, Proposition
\ref{projective-pushforward} calculates the pushforward  $q_!$ as in Section \ref{sec:mflexes-of-lambdalines}, and we get that the degree of the curve of tangent points of bitangent
lines is
\[(d-3)(2d^2+5d-6), \]
in particular, for quintics the degree is 46.

\end{example}

\begin{example}
   The degree of the curve of flex points in a pencil of degree $d$ curves is calculated e.g. in \cite[p. 407]{eh3264}.
\end{example}

\appendix
\section{Transversality} \label{sec:trans}
\subsection{Transversality of a generic section}
\begin{proposition}\label{alg-bertini}(\textbf{Bertini for globally generated bundles}: \cite[Prop. 6.4]{matroid})
    Let $E\to X$ be an algebraic vector bundle, $B$ a vector space and $\varphi:B\to \Gamma(E)$ is a linear family of algebraic sections. Suppose that $E$ is generated by the sections $\varphi(b), b\in B$, i.e. $\Phi(b,x):=\varphi(b)(x):B\times X\to E$ is surjective. Let $Y$ be a subvariety of the total space $E$. Then there is an open subset $U$ of $B$   such that for all $b\in U$ the section $\varphi(b)$ is transversal to $Y$.
\end{proposition}

\begin{example}
For our family of algebraic sections
\[ \varphi: \Pol^{d}(\mathbb{C}^n) \to \Gamma(\Pol^{d}(S)),\, f \mapsto \sigma_f \]
$\Pol^{d}(S)$ is clearly generated by the sections $\sigma_f$, hence for a generic $f \in
\Pol^d(\mathbb{C}^n)$  $\sigma_f$ is transversal to any subvariety $\,\overline Y_\lambda(\Pol^d(S))$
from Section \ref{sec:variety-of-tangent}.
\end{example}

\subsection{Transversality of the universal section} \label{sec:unitrans} The following is a modification of the idea of \cite[Prop. 8.11]{guang}.
We show that, with the proper definitions, universal sections are transversal.

\begin{definition}
  Let $V$ be a $G$-vector space and assume that $j:P\hookrightarrow V$ is an open $G$-invariant subset such that $\pi:P\to P/G$ is a principal $G$-bundle over the smooth $M:=P/G$. Let $W$ be another $G$-vector space and let $\vartheta:V\to W$ be a $G$-equivariant linear map. Then $\vartheta\circ j:P\to W$ is $G$-equivariant, therefore
\[ \sigma_\vartheta: M \to P \times_G W,\, [p] \mapsto \left[ p, \vartheta(p) \right] \]
 determines a section of the associated bundle $P\times_G W$. We call $\sigma_\vartheta$ the \emph{universal section of $\vartheta$}.
\end{definition}
\begin{proposition} \label{universal-vartheta=>transversal} If $\vartheta:V\to W$ is surjective, then the universal section $\sigma_\vartheta$ is transversal to $P\times_G Z$ for any $G$-invariant constructible subset $Z\subset W$.
\end{proposition}
\begin{proof}
The question is local, so let $\varphi:U\to P$ be a local slice to $P$, i.e. $U\subset \R^n$ for $n=\dim_\R(M)$ and $\varphi$ is transversal to the fibers of $\pi$. Then $\pi\varphi:U\to M$ is a chart of $M$. In this local trivialization $\sigma_\vartheta$ has a particularly simple form: If $m=\pi\varphi(u)$ for some $u\in U$, then---denoting the section in this trivialization by $\tilde\sigma_\vartheta$---we have
  \[\tilde\sigma_\vartheta(u)=\big(u,\vartheta(\varphi(u))\big).\]
Therefore, it is enough to show that $\vartheta\varphi$ is transversal to all $G$-orbits. For $\varphi$ this holds by the definition of the local slice. Composing with the surjective, $G$-equivariant and linear $\vartheta$ preserves transversality to the $G$-orbits.
\end{proof}

\begin{example}\label{ex:transversality_univsection} In Section \ref{sec:correspondence} we introduced the vector bundle
\[A_u:=\Hom\big(L,\pol^d(S)\big)\ \ \to \ \ \P\big(\Pol^d(\C^n)\big)\times \gr_2(\C^n)\]
and its section
\[\sigma([f],W)(f):=f|_W,\]
where $L$ and $S$ are the tautological bundles over $\P\big(\Pol^d(\C^n)\big)$
and $\gr_2(\C^n)$.

This is a special case of the previous construction with the choices $V=\Pol^d(\C^n)\oplus \Hom(\C^2,\C^n)$, $W=\Pol^d(\C^2)$, $P=\big(\Pol^d(\C^n)\setminus\{0\}\big)\times\Sigma^0(\C^2,\C^n)$,
$G=\GL(1)\times \GL(2)$ and
  \[\vartheta(f,\beta)=f\circ\beta.\]
The surjectivity of $\vartheta$ is clear, so the universal section is transversal to the $\,\overline Y_\lambda(A_u)$'s, therefore our calculation for the cohomology  class of $[\sigma^{-1}(\,\overline{Y}_\lambda(A_u))]$ is valid.
\end{example}

\bibliography{plucker}
\bibliographystyle{alpha}
\end{document}